\pdfoutput=1
\RequirePackage{silence}
\documentclass[a4paper,svgnames]{amsart}
\usepackage[hmarginratio={1:1},vmarginratio={1:1},lmargin=60.0pt,tmargin=60.0pt]{geometry}

\usepackage{booktabs}
\allowdisplaybreaks

\usepackage[T1]{fontenc}
\usepackage{latexsym,exscale,amsfonts,amssymb,mathtools}
\usepackage{amsmath,amsthm,amsfonts,amssymb,amscd,textcomp,bbm}
\usepackage{mathrsfs,stackrel}
\usepackage{etoolbox}
\usepackage{tikz,tikz-cd}
\usetikzlibrary{matrix,arrows.meta,decorations.markings,shapes.multipart,decorations.pathreplacing,decorations.pathmorphing}
\usepackage{aliascnt}
\usepackage{ytableau}
\usepackage{xparse}
\usepackage{cite}

\usepackage{dynkin-diagrams}

\usepackage{array}
\newcolumntype{C}{>{$}c<{$}}

\usepackage{xcolor,colortbl}

\definecolor{mygray}{gray}{0.6}
\definecolor{mygraydark}{gray}{0.4}
\definecolor{mygraylight}{gray}{0.85}
\definecolor{spinach}{RGB}{46,139,87}
\definecolor{tomato}{RGB}{255,99,71}
\definecolor{orchid}{RGB}{143,40,194}
\definecolor{neon}{RGB}{77,77,255}
\definecolor{pumpkin}{RGB}{224,180,80}
\definecolor{citron}{RGB}{190,180,90}

\definecolor{lava}{RGB}{207,16,32}
\definecolor{cream}{RGB}{255,253,208}
\definecolor{verdigris}{RGB}{67,179,174}
\definecolor{Black}{RGB}{0,0,0}
\definecolor{mydarkblue}{RGB}{10,10,170}
\definecolor{darkspinach}{RGB}{20,70,20}
\definecolor{darktomato}{RGB}{155,40,30}
\definecolor{darkorchid}{RGB}{50,10,100}
\definecolor{darklava}{RGB}{150,8,16}

\usepackage{todonotes}

\usepackage{enumitem}
\setlist[enumerate]{itemsep=0.15cm,label=\emph{\upshape(\alph*)}}
\setlist[enumerate,2]{itemsep=0.15cm,label=\emph{\upshape(\roman*)}}
\setlist[enumerate,3]{itemsep=0.15cm,label=\emph{\upshape(\Alph*)}}

\let\emph\relax
\DeclareTextFontCommand{\emph}{\em}

\renewcommand{\dots}{\text{...}}

\newcommand{\eg}{\text{e.g.}}

\newcommand{\C}{\mathbb{C}}
\newcommand{\R}{\mathbb{R}}
\newcommand{\Rplus}{\mathbb{R}_{\geq 0}}
\newcommand{\N}{\mathbb{Z}_{\geq 0}}

\newcommand{\Z}{\mathbb{Z}}
\newcommand{\K}{\mathbb{K}}
\newcommand{\F}[1][p]{\mathbb{F}_{#1}}

\newcommand{\setstuff}[1]{\mathrm{#1}}
\newcommand{\catstuff}[1]{\mathbf{#1}}

\font\scc=rsfs10
\newcommand{\twocatstuff}[1]{\scc\mbox{#1}\hspace{1.0pt}}

\newcommand{\obstuff}[1]{\mathtt{#1}}

\newcommand{\mt}{\textit{MT}}

\newcommand{\pfdim}{\mathrm{PFdim}\,}
\newcommand{\aint}{\overline{\Z}}
\newcommand{\mainsymbol}{a}

\newcommand{\rep}{\catstuff{Rep}}
\newcommand{\proj}{\catstuff{Proj}}
\newcommand{\inj}{\catstuff{Inj}}
\newcommand{\catgen}[1][\obstuff{X}]{\{#1^{\otimes d}|d\in\N\}^{\oplus,\subset_{\oplus}}}
\newcommand{\sbim}{\catstuff{SBim}}
\newcommand{\verlinde}[1][k]{\catstuff{Ver}_{#1}(\mathrm{SL}_{2})}
\newcommand{\verlindeg}[1][k]{\catstuff{Ver}_{#1}(\mathrm{SL}_{3})}

\newcommand{\simples}[1]{\setstuff{S}(#1)}
\newcommand{\sltwo}[1][p]{\mathrm{SL}_{2}(\mathbb{F}_{#1})}

\newcommand{\rmat}[1][m]{\mathrm{Mat}_{#1}(\Rplus)}

\usepackage[all]{xy}
\usepackage{tikz}
\usetikzlibrary{cd}
\usetikzlibrary{decorations}
\usetikzlibrary{decorations.markings}
\usetikzlibrary{decorations.pathreplacing}
\usetikzlibrary{decorations.pathmorphing}
\usetikzlibrary{arrows.meta,shapes,positioning,matrix,calc}
\usetikzlibrary{shapes.callouts}
\tikzset{anchorbase/.style={baseline={([yshift=-0.5ex]current bounding box.center)}},
tinynodes/.style={font=\tiny,text height=0.25ex,text depth=0.05ex},
smallnodes/.style={font=\scriptsize,text height=0.75ex,text depth=0.15ex},
}

\numberwithin{equation}{subsection}

\newtheorem{Theorem}{Theorem}

\AtEndEnvironment{Corollary}{\null\hfill$\square$}%
\newtheorem{Lemma}{Lemma}

\theoremstyle{definition}

\AtEndEnvironment{Definition}{\null\hfill$\Diamond$}%
\newtheorem{Example}{Example}
\AtEndEnvironment{Example}{\null\hfill$\Diamond$}%

\theoremstyle{remark}

\newtheorem{Remark}{Remark}
\AtEndEnvironment{Remark}{\null\hfill$\Diamond$}%
\newtheorem{Question}{Question}
\AtEndEnvironment{Question}{\null\hfill$\Diamond$}%

\usepackage[hypertexnames=false]{hyperref}
\usepackage{bookmark}
\hypersetup{
pdftoolbar=true,
pdfmenubar=true,
pdffitwindow=false,
pdfstartview={FitH},
pdftitle={Asymptotics in finite monoidal categories},
pdfauthor={No one},
pdfsubject={},
pdfcreator={No one},
pdfproducer={No one},
pdfkeywords={},
pdfnewwindow=true,
colorlinks=true,
linkcolor=mydarkblue,
citecolor=teal,
filecolor=magenta,
urlcolor=orchid,
linkbordercolor=lava,
citebordercolor=teal,
urlbordercolor=orchid,
linktocpage=true
}

\newcommand{\nnfootnote}[1]{%
\begin{NoHyper}
\renewcommand\thefootnote{}\footnote{#1}%
\addtocounter{footnote}{-1}%
\end{NoHyper}
}

\def\makeautorefname#1#2{\csdef{#1autorefname}{#2}}
\makeautorefname{section}{Section}%
\makeautorefname{subsection}{Section}%
\makeautorefname{subsubsection}{Section}%

\makeautorefname{Theorem}{Theorem}%
\makeautorefname{Proposition}{Proposition}%
\makeautorefname{Corollary}{Corollary}%
\makeautorefname{Lemma}{Lemma}%

\makeautorefname{Definition}{Definition}%
\makeautorefname{Example}{Example}%

\makeautorefname{Remark}{Remark}%
\makeautorefname{Question}{Question}%

\begin{document}
\title[Asymptotics in finite monoidal categories]{Asymptotics in finite monoidal categories}
\author[A. Lacabanne, D. Tubbenhauer and P. Vaz]{Abel Lacabanne, Daniel Tubbenhauer and Pedro Vaz}

\address{A.L.: Laboratoire de Math{\'e}matiques Blaise Pascal (UMR 6620), Universit{\'e} Clermont Auvergne, Campus Universitaire des C{\'e}zeaux, 3 place Vasarely, 63178 Aubi{\`e}re Cedex, France,\newline \href{http://www.normalesup.org/~lacabanne}{www.normalesup.org/$\sim$lacabanne},
\href{https://orcid.org/0000-0001-8691-3270}{ORCID 0000-0001-8691-3270}}
\email{abel.lacabanne@uca.fr}

\address{D.T.: The University of Sydney, School of Mathematics and Statistics F07, Office Carslaw 827, NSW 2006, Australia, \href{http://www.dtubbenhauer.com}{www.dtubbenhauer.com}, \href{https://orcid.org/0000-0001-7265-5047}{ORCID 0000-0001-7265-5047}}
\email{daniel.tubbenhauer@sydney.edu.au}

\address{P.V.: Institut de Recherche en Math{\'e}matique et Physique, 
Universit{\'e} catholique de Louvain, Chemin du Cyclotron 2,  
1348 Louvain-la-Neuve, Belgium, \href{https://perso.uclouvain.be/pedro.vaz}{https://perso.uclouvain.be/pedro.vaz}, \href{https://orcid.org/0000-0001-9422-4707}{ORCID 0000-0001-9422-4707}}
\email{pedro.vaz@uclouvain.be}

\begin{abstract}
We give explicit formulas for the asymptotic growth rate of 
the number of summands in tensor powers in certain monoidal categories 
with finitely many indecomposable objects, and related structures.
\end{abstract}

\nnfootnote{\textit{Mathematics Subject Classification 2020.} Primary: 11N45, 18M05; Secondary: 16T05, 18M20, 26A12.}
\nnfootnote{\textit{Keywords.} Tensor products, asymptotic behavior, monoidal categories.}

\addtocontents{toc}{\protect\setcounter{tocdepth}{1}}

\maketitle

\tableofcontents

%%%%%%%%%%%%%%%%%%%%%%%%%%%%%%%%%%%%%%%%%

\section{Introduction}\label{S:Intro}

%%%%%%%%%%%%%%%%%%%%%%%%%%%%%%%%%%%%%%%%%

\addtocounter{subsection}{1}

Let $R=(R,C)$ be a \emph{finite based $\Rplus$-algebra} with basis $C=\{1=c_{0},\dots,c_{r-1}\}$ (recalled in \autoref{S:Proof} together with some other notions used in this introduction). 
Recall that we thus have
\begin{gather}\label{Eq:IntroBased}
c_{i}c_{j}=\sum_{k}m_{i,j}^{k}\cdot c_{k}
\quad\text{with}\quad
m_{i,j}^{k}\in\Rplus.
\end{gather}
Iterating this gives us coefficients $m_{i,j,\dots,l}^{k}\in\Rplus$. 
Similarly, for $c=a_{0}\cdot c_{0}+\dots+a_{r-1}\cdot c_{r-1},d=d_{0}\cdot c_{0}+\dots+d_{r-1}\cdot c_{r-1}\in\Rplus C$ we get, for example, $cd=\sum_{k}a_{i}d_{j}m_{i,j}^{k}\cdot c_{k}$ with 
$a_{i}d_{j}m_{i,j}^{k}\in\Rplus$.

Fix $c\in\Rplus C$.
We write $m_{n}^{\ast}(c)$ for these coefficients as 
they appear in $c^{n}$ where $\ast\in\{0,\dots,r-1\}$.
Define
\begin{gather*}
b_{n}^{R,c}:=\sum_{\ast}m_{n}^{\ast}(c)=\text{total sum of the coefficients $m_{n}^{\ast}(c)$}.
\end{gather*}
Moreover, we define the function
\begin{gather*}
b^{R,c}\colon\N\to\Rplus,n\mapsto b^{R,c}(n):=b_{n}^{R,c}.
\end{gather*}
We are interested in the \emph{asymptotic behavior} of the function 
$b^{R,c}(n)$. The main question we address is:

\begin{Question}\label{Q:IntroMain}
Find an explicit formula $\mainsymbol(n)$ such that
\begin{gather*}
b^{R,c}(n)\sim\mainsymbol(n),
\end{gather*}
where we write $\sim$ for asymptotically equal.
\end{Question}

We answer \autoref{Q:IntroMain} as follows.

The (transposed) \emph{action matrix} of $c=a_{0}\cdot c_{0}+\dots+a_{r-1}\cdot c_{r-1}\in\Rplus C$ is the matrix $(\sum_{i}a_{i}m_{i,j}^{k})_{k,j}$. Abusing language, we 
will call the 
submatrix of it corresponding to the connected component of $1$ 
also the action matrix and use this below.

Assume that the \emph{Perron--Frobenius theorem} holds, that is
the action matrix of $c\in\Rplus C$ has a leading eigenvalue $\lambda_{0}=\pfdim c$ of multiplicity one that we call 
the \emph{Perron--Frobenius dimension} of $c$. 
Moreover, the action matrix 
has some period $h\in\N$ such that $\lambda_{k}=\zeta^{k}\pfdim c$, where $\zeta=\exp(2\pi i/h)$ and $k\in\{1,\dots,h-1\}$,  
are precisely the other eigenvalues of absolute value $\pfdim c$. We will drop this assumption in \autoref{S:Proof} below.

Let us denote the right (the one with for $Mv_{i}=\lambda_{i}\cdot v_{i}$) and left
(the one with $w^{T}_{i}M=\lambda_{i}\cdot w^{T}_{i}$) eigenvectors by 
$v_{i}$ and $w_{i}$, normalized such that $w^{T}_{i}v_{i}=1$.
Let $v_{i}w_{i}^{T}[1]$ denote taking the sum of the first column of the matrix 
$v_{i}w_{i}^{T}$, and let $\aint$ denote the algebraic integers.
Define 
\begin{gather}\label{Eq:IntroMainSymbol}
\mainsymbol(n)=
\big(v_{0}w_{0}^{T}[1]\cdot 1+v_{1}w_{1}^{T}[1]\cdot\zeta^{n}+v_{2}w_{2}^{T}[1]\cdot(\zeta^{2})^{n}+\dots+v_{h-1}w_{h-1}^{T}[1]\cdot(\zeta^{h-1})^{n}\big)
\cdot(\pfdim c)^{n}\in\aint.
\end{gather}
Let $\lambda^{sec}$ be any second largest eigenvalue of the action matrix of $c$.
We will prove (see \autoref{S:Proof} below):

\begin{Theorem}\label{T:IntroMain}
We have
\begin{gather*}
b^{R,c}(n)\sim\mainsymbol(n),
\end{gather*}
and the convergence is geometric with ratio $|\lambda^{sec}/\pfdim c|$. In particular,
\begin{gather*}
\beta^{R,c}:=\lim_{n\to\infty}\sqrt[n]{b_{n}^{R,c}}=\pfdim c.
\end{gather*}
\end{Theorem}

The reason why \autoref{T:IntroMain} is interesting from the 
categorical point of view is the following.
For us a \emph{finite monoidal category}
is a category such that:
\begin{enumerate}[label=(\roman*)]

\item It is monoidal.

\item It is additive Krull--Schmidt.

\item It has finitely many (isomorphism classes of) indecomposable objects.

\end{enumerate}

\begin{Example}\label{E:IntroFiniteCats}
Here are a few examples:
\begin{enumerate}

\item Let $G$ be a finite group and consider $\rep(G)=\rep(G,\C)$ the category of finite dimensional complex representations of $G$. This is a prototypical example 
of a finite monoidal category.

\item More generally, all fusion categories are finite monoidal.

\item For a finite group $G$, and arbitrary field, we can consider 
finite dimensional projective $\proj(G)$ or injective $\inj(G)$ representations. These are 
finite monoidal categories. More generally, one can take any finite dimensional Hopf algebra instead of a finite group.

\item If we assume that a Hopf algebra $H$ is of finite type, then we can even consider $\rep(H)$ (finite dimensional $H$-representations). An explicit and nonsemisimple example over $\C$ is the Taft algebra by \cite[Theorem 2.5]{ChVaOyZh-green-taft}.

\item In any additive Krull--Schmidt monoidal category one can take
$\catgen$, the additive idempotent completion of the full subcategory generated by an object $\obstuff{X}$, as long as this has finitely 
many indecomposable objects. Explicitly, for a finite group $G$ one can take any two dimensional $G$-representation for $\obstuff{X}$, which follows from 
\cite{Al-proj-sl2}. There are 
many more examples, see {\eg} \cite{Cr-tensor-simple-modules}.

\item Consider Soergel bimodules $\sbim(W)$ as in \cite{So-hcbim}.
These are finite monoidal categories if $W=(W,S)$ is of finite Coxeter type.

\end{enumerate}
There are of course many more examples.
\end{Example}

The following is very easy and omitted:

\begin{Lemma}\label{L:IntroGrothendieck}
The additive Grothendieck ring of a finite monoidal category
is a finite based $\Rplus$-algebra with basis given 
by the classes of indecomposable objects.
\qed
\end{Lemma}

Fix a finite monoidal category 
$\catstuff{C}$ and an object $\obstuff{X}\in\catstuff{C}$.
Following \cite{CoOsTu-growth}, we define 
\begin{gather*}
b_{n}^{\catstuff{C},\obstuff{X}}:=\#\text{indecomposable summands in $\obstuff{X}^{\otimes n}$ counted with multiplicities}.
\end{gather*}
Note that $\mainsymbol(n)$ has an analog in this context, denoted by the same symbol, obtained for 
the (transposed) action matrix for left tensoring. Similarly as before we also have $\lambda^{sec}$.
We then get:

\begin{Theorem}\label{T:IntroMainTwo}
Under the same assumption as in \autoref{T:IntroMain}, we have
\begin{gather*}
b^{\catstuff{C},\obstuff{X}}(n)\sim\mainsymbol(n),
\end{gather*}
and the convergence is geometric with ratio $|\lambda^{sec}/\pfdim\obstuff{X}|$. 
In particular,
\begin{gather*}
\beta^{\catstuff{C},\obstuff{X}}:=\lim_{n\to\infty}\sqrt[n]{b_{n}^{\catstuff{C},\obstuff{X}}}=\pfdim\obstuff{X}.
\end{gather*}
\end{Theorem}

\begin{proof}
From \autoref{T:IntroMain} and \autoref{L:IntroGrothendieck}.
\end{proof}

In the next section we will discuss examples of \autoref{T:IntroMainTwo}, 
and then we will prove \autoref{T:IntroMain}. 
We also generalize these two theorems in \autoref{S:Proof} by getting rid of 
the assumption on the action matrix.
Before that, 
let us finish the introduction with some (historical) remarks.

\begin{Remark}\label{R:IntroHistory}
\leavevmode

\begin{enumerate}

\item To study asymptotic properties of tensor powers is a rather new 
subject and most things are still quite mysterious. 
Let us mention a few facts that are known.
An early reference we know is \cite{Bi-asymptotic-lie}, which studies 
questions similar to the one in this note but for Lie algebras, and this was carried 
on in several works such as \cite{PoRe-mult-large-tensor-powers}.
As another example, the paper \cite{BeSy-non-projective-part} studies the 
growth rate of the dimensions of the non-projective part of tensor powers 
of a representation of a finite group. More generally, the paper \cite{CoEtOs-frobenius-exact}
studies, working in certain tensor categories, the growth rates of summands of categorical dimension prime to the underlying characteristic. The paper \cite{CoOsTu-growth} 
studies the growth rate of all summands, while \cite{KhSiTu-monoidal-cryptography}
studies the Schur--Weyl dual question.

\item \autoref{T:IntroMain} and \autoref{T:IntroMainTwo} 
generalize \cite[Proposition 2.1]{CoEtOs-growth-mod-p}. And for us one of the main features of that proposition is it simplicity, having a 
simple statement and proof. As we will see, the same is true 
for \autoref{T:IntroMain} and \autoref{T:IntroMainTwo} as well: Clearly, 
the statements themselves are (surprisingly) simple yet general. Moreover,  
the proof of \autoref{T:IntroMain}, and therefore the proof 
of \autoref{T:IntroMainTwo} as well, is rather straightforward 
as soon as the key ideas are in place. 

\item The second statements in \autoref{T:IntroMain} and \autoref{T:IntroMainTwo} were already 
observed in \cite{CoOsTu-growth} (in the setting of \cite{CoOsTu-growth} the Perron--Frobenius dimension agrees with the usual dimension), but the (finer) asymptotic behavior appears to be new.

\end{enumerate}
Finally, let us mention that similar 
questions have been studied much earlier, see for example \cite{AlEv-representations-quillen} for 
a related notion involving length of projective resolutions, or \cite{LoSh-random-young}
for counting and Young diagrams.
\end{Remark}

\noindent\textbf{Acknowledgments.}
We like to thank Kevin Coulembier, 
Pavel Etingof and Victor Ostrik for 
very helpful email exchanges, and the referee for a careful reading of our document. DT thanks randomness for 
giving them/us the key idea underlying this note.

This project was in part supported by Universit{\'e} Clermont Auvergne 
and Universit{\'e} catholique de Louvain, which is gratefully acknowledged.
DT was supported by the Australian research council, and
PV was supported by the Fonds de la Recherche Scientifique-FNRS under Grant no. CDR-J.0189.23.

%%%%%%%%%%%%%%%%%%%%%%%%%%%%%%%%%%%%%%%%%

\section{Examples}\label{S:Examples}

%%%%%%%%%%%%%%%%%%%%%%%%%%%%%%%%%%%%%%%%%

Let us call \autoref{T:IntroMain} and \autoref{T:IntroMainTwo} 
our \emph{main theorem(s)} or {\mt} for short.
To underpin the explicit nature of these theorems, we now list examples 
{\mt} applies. We also add that all the below 
can be double checked using the code on \cite{LaTuVa-code-growth}. 
That page also contains a (potentially empty) Erratum.

Let us briefly explain why {\mt} 
can be used in all the examples discussed below:
For \autoref{SS:ExamplesFiniteGroups} this follows since our assumption 
on $V$ implies that the action matrix is irreducible.
For all other examples a direct calculation verifies that the action matrices 
satisfy the Perron--Frobenius theorem.

%%%%%%%%%%%%%%%%%%%%%%%%%%%%%%%%%%%%%%%%%

\subsection{Finite groups}\label{SS:ExamplesFiniteGroups}

%%%%%%%%%%%%%%%%%%%%%%%%%%%%%%%%%%%%%%%%%

Let $G$ be a finite group. Given a finite dimensional complex $G$-representation 
$V$, denote by $Z_{V}(G)\subset G$ the 
subgroup consisting of elements of $g$ that acts as a scalar on $V$ and by 
$\omega_{V}(g)\in\C$ the corresponding scalar. If $V$ is simple, 
then $\omega_{V}$ is known as the \emph{central character} of $V$. 

Suppose that $V$ is a faithful $G$-representation. 
Since $V$ is faithful we get that $Z_{V}(G)$ is a subgroup of $Z(G)$ 
and also that the action graph of tensoring with $V$ is connected (in the oriented sense). Then {\mt} implies:
\begin{gather}\label{Eq:ExamplesFiniteGroups}
\mainsymbol(n)= 
\left(\frac{1}{\#G}
\sum_{g\in Z_{V}(G)}
\Big(\sum_{L\in\simples{G}}\omega_L(g)\dim_{\C}L\Big)\cdot
\omega_{V}(g)^{n}\right)\cdot(\dim_{\C}V)^{n},
\end{gather}
where $\simples{G}=\{\text{simple $G$-representations}\}/\cong$.
This follows directly from {\mt} after recalling the 
connection from Perron--Frobenius theory to character theory 
as explained in {\eg} \cite[Chapter 3 and Example 4.5.5]{EtGeNiOs-tensor-categories}. To elaborate a bit, the Perron--Frobenius dimension in this case is just the dimension, and the leading eigenvector corresponds to the regular $G$-representation.

\begin{Remark}\label{R:ExamplesFiniteGroups}
If $V$ is not faithful, then the action graph of tensoring with $V$ needs not to be connected, 
but that is not an issue in {\mt}.
Thus, the assumption that $V$ is faithful can be easily relaxed.
\end{Remark}

\begin{Remark}\label{R:ExamplesFiniteGroupsTwo}
Alternatively one can prove \eqref{Eq:ExamplesFiniteGroups} using character theory, similarly to 
\cite[Proposition 2.1]{CoEtOs-growth-mod-p}. \eqref{Eq:ExamplesFiniteGroups} still generalizes
\cite[Proposition 2.1]{CoEtOs-growth-mod-p}.
\end{Remark}

Let us give a few explicit examples.

\begin{Example}[Dihedral groups]\label{E:ExamplesDihedral}
Let $m\in\Z_{\geq 3}$ and let $G$ be 
the \emph{dihedral group} of order $2m$. 
Let $m^{\prime}=m/2$, if $m$ is even, and 
$m^{\prime}=(m-1)/2$, if $m$ is odd. Choose $V$ any faithful 
representation of dimension $2$ of $G$. Then \eqref{Eq:ExamplesFiniteGroups} 
gives the formulas
\begin{gather*}
\mainsymbol(n) =
\begin{cases}
\frac{m+1}{2m}\cdot 2^{n} & \text{if } m \text{ is odd},
\\
\frac{m+2}{2m}\cdot 2^{n} & \text{if } m \text{ is even and }m^{\prime}\text{ is odd},
\\
\left(\frac{(m+2)}{2m}\cdot 1+\frac{1}{m}\cdot(-1)^{n}\right)\cdot 2^{n} & \text{if }m\text{ is even and } m^{\prime}\text{ is even}.
\end{cases}
\end{gather*}
Two explicit examples are $m\in\{4,5\}$ and $V$ is the $G$-representation corresponding to 
rotation by $2\pi/m$. Then:
\begin{gather*}
m=4\colon
\begin{tikzpicture}[anchorbase]
\node at (0,0) {\includegraphics[height=2.7cm]{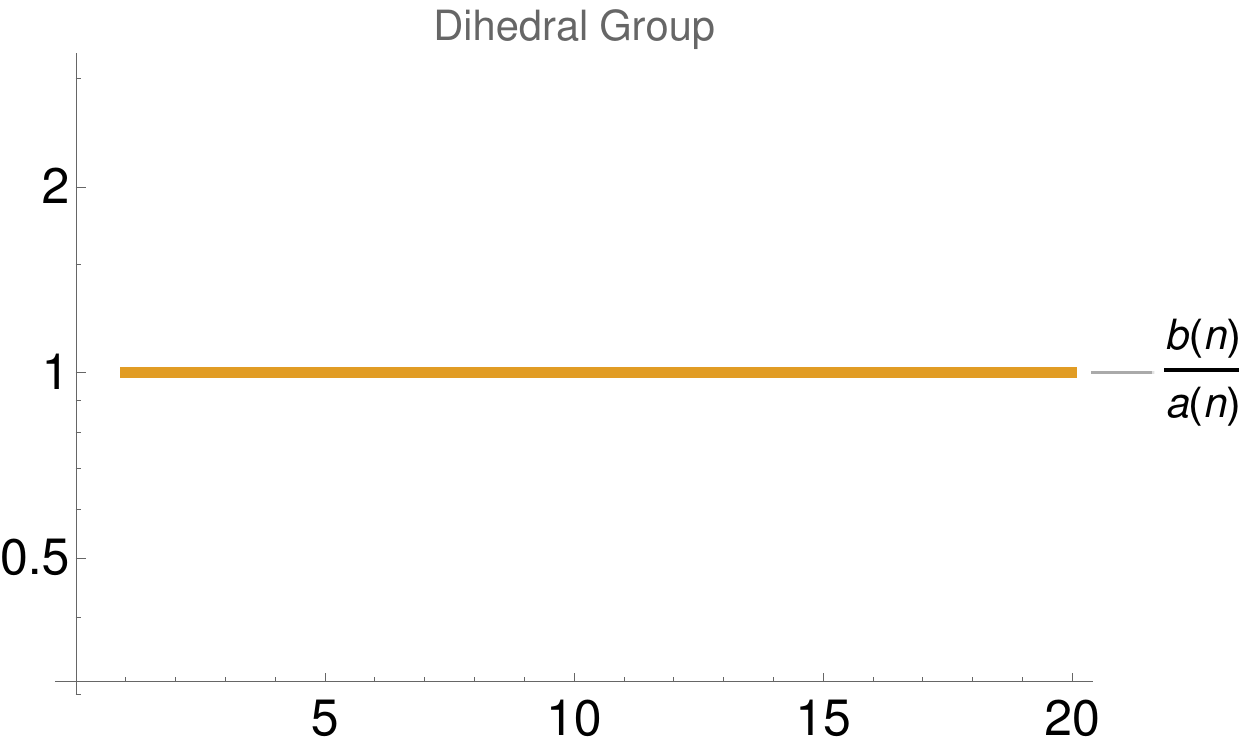}};
\end{tikzpicture}
,\quad
m=5\colon
\begin{tikzpicture}[anchorbase]
\node at (0,0) {\includegraphics[height=2.7cm]{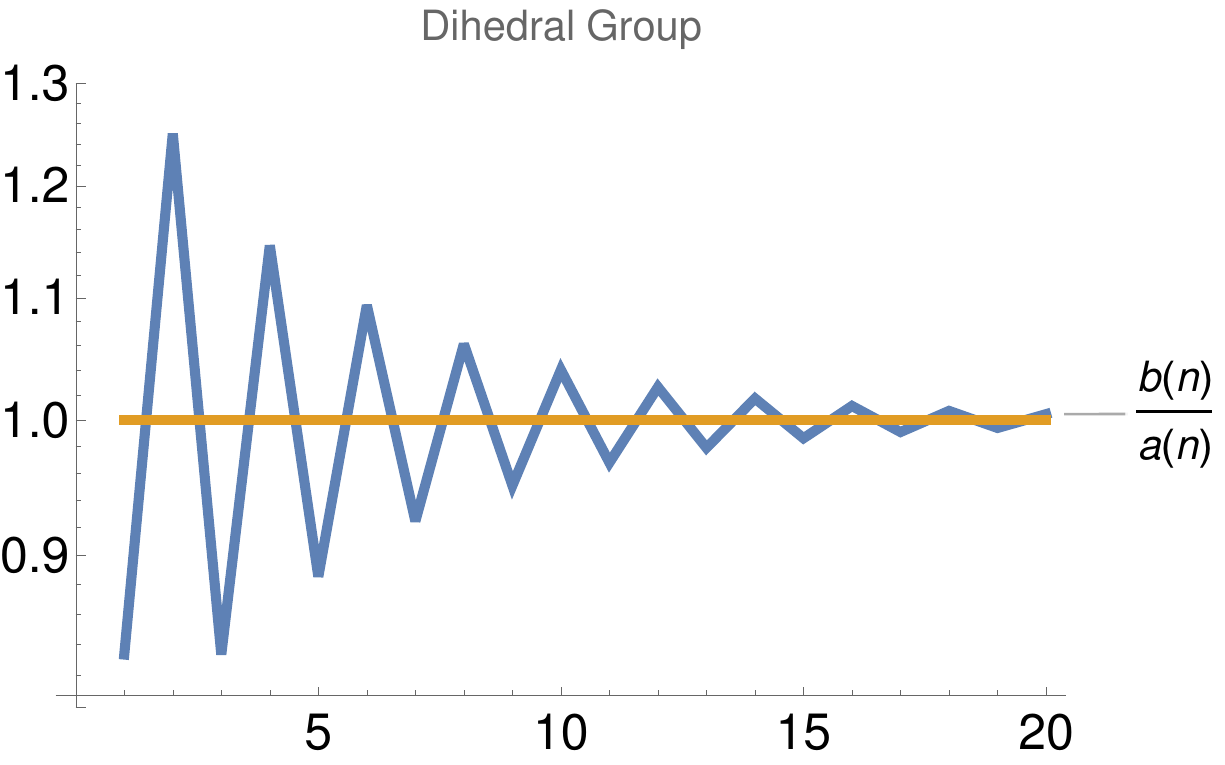}};
\end{tikzpicture}
.
\end{gather*}
Here and throughout, we display the graphs of $b(n)/\mainsymbol(n)$ in the usual way but log plotted (on the $y$-axis). Moreover, for $m=4$ we have $b(n)=\mainsymbol(n)$ and we will omit plots in case that happens.
\end{Example}

The next example can be seen as a $p>2$ version of \autoref{E:ExamplesDihedral}.

\begin{Example}[Extraspecial groups]\label{E:ExamplesExtra}
Let $p$ be a prime and $m\in\Z_{\geq 1}$. 
Recall that a $p$-group of order $p^{1+2m}$ is called 
\emph{extraspecial} if its center $Z(G)$ is of order 
$p$ and the quotient $G/Z(G)$ is a $p$-elementary 
abelian group. For each $p$ and $m$, there exists 
two isomorphism classes of extraspecial groups of 
order $p^{1+2m}$, and they have the same character table. Thus, 
by \eqref{Eq:ExamplesFiniteGroups} we can take any of these two without difference.
In the special case $p=2$ and 
$m=1$ we recover the dihedral group and the 
quaternion group of order $8$.

Fix now an extraspecial group $G$ of order $p^{1+2m}$. The simple $G$-representations are given as follows:
\begin{enumerate}[label=(\roman*)]

\item There are $p^{2m}$ nonisomorphic one dimensional representations that arise from the representations of $G/Z(G)$.

\item There are $p-1$ nonisomorphic simple representations of dimension $p^{m}$ which are characterized by their central character.

\end{enumerate}
Choose $V$ any of the simple $G$-representation of dimension $p^{m}$. Then $Z_{V}(G)=Z(G)$ and
\eqref{Eq:ExamplesFiniteGroups} gives
\begin{gather*}
\mainsymbol(n)= 
\begin{cases}
(p^m)^{n} & \text{if }p\mid n, 
\\
(p^m)^{n-1} & \text{otherwise}.
\end{cases}
\end{gather*}
It turns out that this formula is not only asymptotic: we have $b(n)=\mainsymbol(n)$. 
This is due to the fact that the character of $V$ vanishes outside of $Z(G)$.
\end{Example}

\begin{Example}[Imprimitive complex reflection groups]\label{E:ExamplesCRG}
Let $d$ and $m$ be integers in $\Z_{\geq 1}$ and consider the \emph{imprimitive complex reflection group} $G=G(d,1,m)$. This group can be seen as the group of $m$-by-$m$ monomial matrices with entries being $d$th roots of unity. For $d=1$ we recover the symmetric group, covered by \cite[Example 2.3]{CoEtOs-growth-mod-p}, and if $d=2$ we recover the Weyl group of type $B_{m}$.

Choose $V$ the standard representation given by the matrix description of $G$. 
Then \eqref{Eq:ExamplesFiniteGroups} gives a formula akin to 
\cite[Example 2.3]{CoEtOs-growth-mod-p} which we decided not to write down as its a bit tedious.

In any case, for the special cases $d\in\{1,2\}$ and $m=2$, or $d=2$ and $m=4$ we get
\begin{gather*}
\left\{
\begin{gathered}
d=1,
\\[-0.1cm]
m=3 
\end{gathered}
\right.
\colon
\mainsymbol(n)=\frac{2}{3}\cdot 3^{n}
,\quad
\left\{
\begin{gathered}
d=2,
\\[-0.1cm]
m=3 
\end{gathered}
\right.
\colon\mainsymbol(n)=\frac{5}{12}\cdot 3^{n}
,\quad
\left\{
\begin{gathered}
d=2,
\\[-0.1cm]
m=4
\end{gathered}
\right.
\colon\mainsymbol(n)=\left(\frac{19}{96}\cdot 1+\frac{1}{32}\cdot (-1)^{n}\right)\cdot 4^{n}
.
\end{gather*}
We get the plots
\begin{gather*}
\left\{
\begin{gathered}
d=2,
\\[-0.1cm]
m=3
\end{gathered}\right.\colon
\begin{tikzpicture}[anchorbase]
\node at (0,0) {\includegraphics[height=2.7cm]{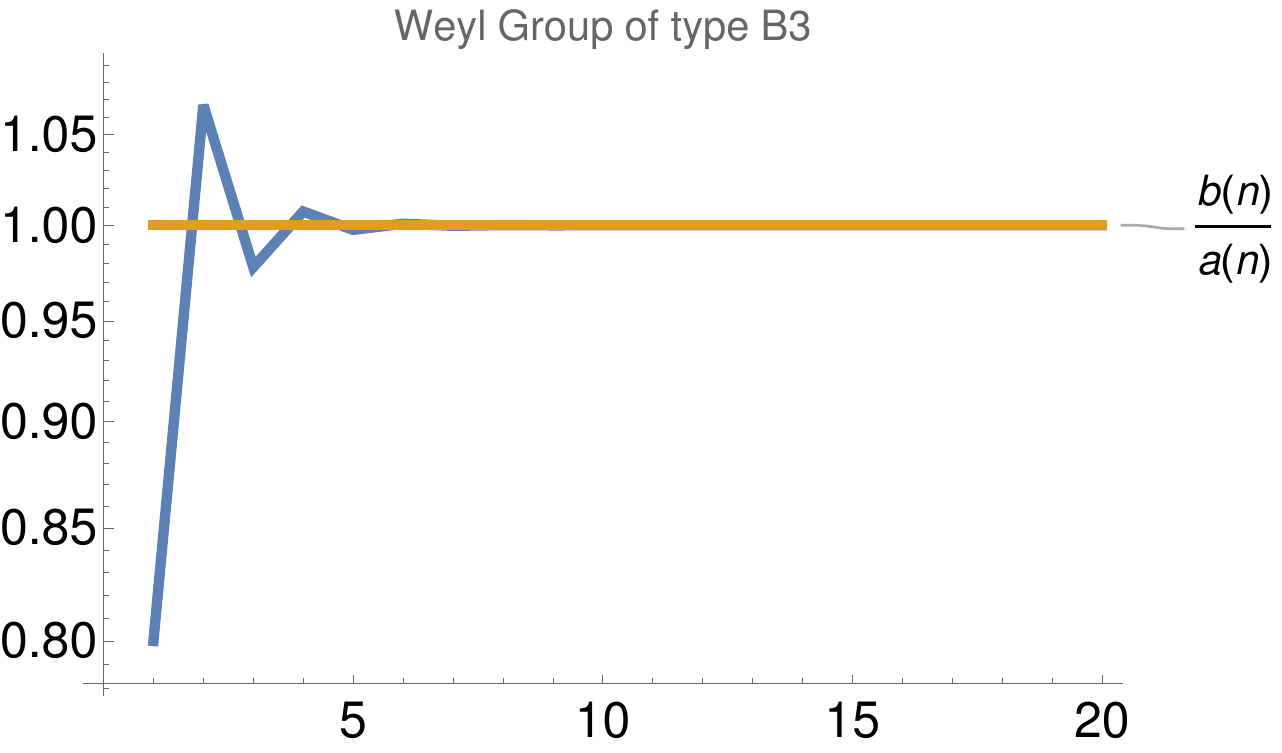}};
\end{tikzpicture}
,\quad
\left\{
\begin{gathered}
d=2,
\\[-0.1cm]
m=4
\end{gathered}\right.\colon
\begin{tikzpicture}[anchorbase]
\node at (0,0) {\includegraphics[height=2.7cm]{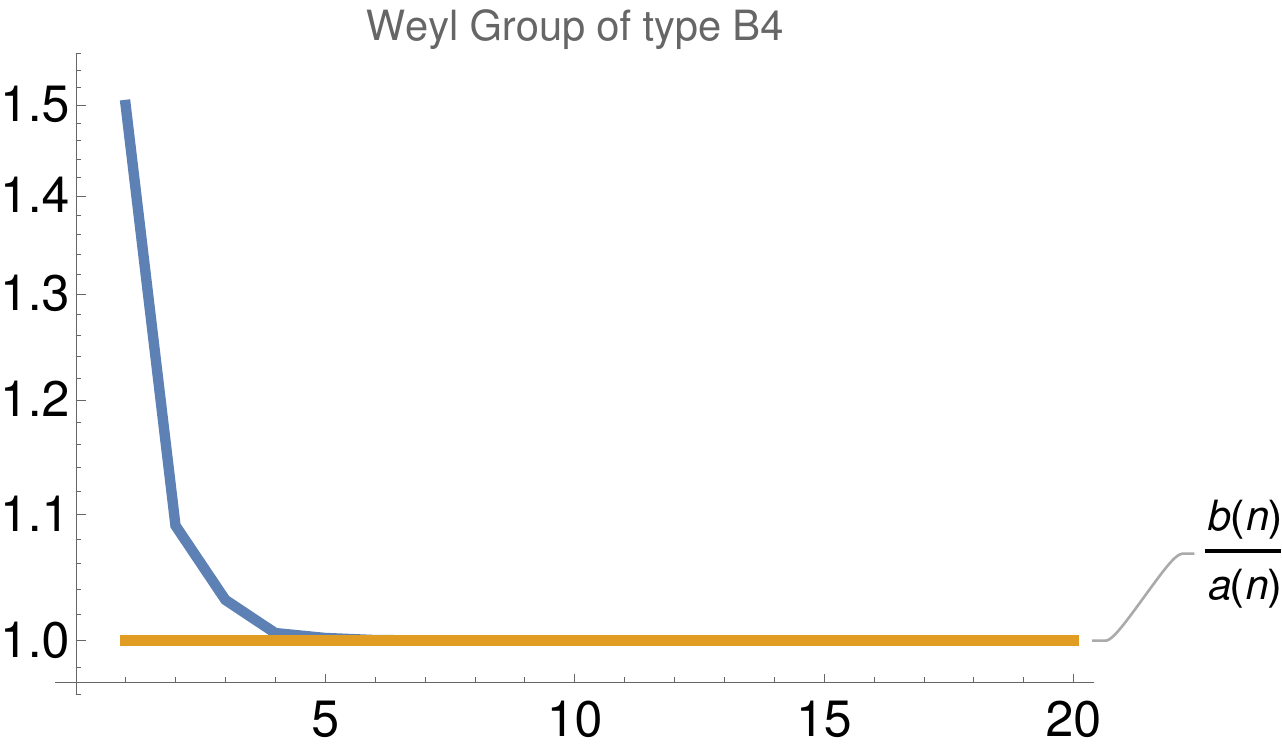}};
\end{tikzpicture}
.
\end{gather*}
Moreover, the formula $\mainsymbol(n)=\frac{2}{3}\cdot 3^{n}$ is exact for $d=1$ and $m=3$.
\end{Example}

%%%%%%%%%%%%%%%%%%%%%%%%%%%%%%%%%%%%%%%%%

\subsection{Fusion categories}\label{SS:ExamplesFusion}

%%%%%%%%%%%%%%%%%%%%%%%%%%%%%%%%%%%%%%%%%

This section discusses fusion categories over $\C$ different from $\rep(G)$.

\begin{Example}[Fibonacci category]\label{E:ExamplesFibonacci}
Let $\twocatstuff{F}$ be the \emph{Fibonacci category}, see for example 
\cite[Exercise 8.18.7]{EtGeNiOs-tensor-categories} where 
$\twocatstuff{F}$ is denoted $\mathcal{YL}_{+}$ 
(or $\mathcal{YL}_{-}$, depending on conventions). All we need to know is that $\twocatstuff{F}$ 
is $\otimes$-generated by one object $\obstuff{X}$ with action matrix
$M(\obstuff{X})=
\begin{psmallmatrix}
0 & 1
\\
1 & 1
\end{psmallmatrix}$.

We want to estimate $b^{\scalebox{0.7}{$\twocatstuff{F}$},\obstuff{X}}(n)$. To this end, 
the eigenvalues of $M(\obstuff{X})$ are the two roots of $x^{2}=x+1$, in particular, 
$\pfdim\obstuff{X}=\phi$, the golden ratio. Its Perron--Frobenius eigenvectors 
are $v=w=\scalebox{0.7}{$\big(\frac{\sqrt{5}-1}{\sqrt{10-2\sqrt{5}}},\sqrt{\frac{1}{10}(\sqrt{5}+5)}\big)^{T}$}$ and we therefore get
\begin{gather*}
\mainsymbol(n)=
\frac{1}{10}(\sqrt{5}+5)\cdot\phi^{n}=\frac{1}{\sqrt{5}}\phi^{n-1},
\quad
\begin{tikzpicture}[anchorbase]
\node at (0,0) {\includegraphics[height=2.7cm]{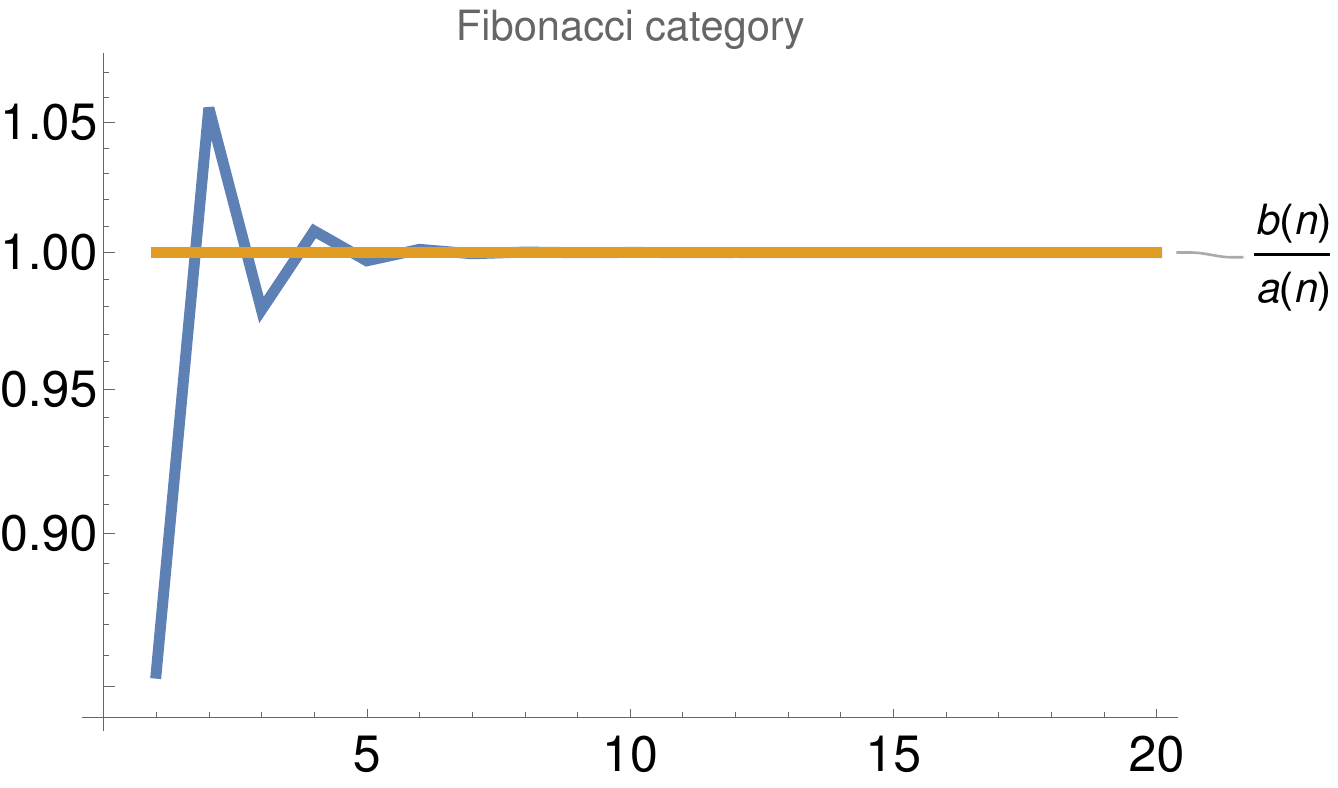}};
\end{tikzpicture}
,
\end{gather*}
from {\mt}. Note that the classical asymptotic for the Fibonacci numbers is $\frac{1}{\sqrt{5}}\phi^{n}$ and not $\frac{1}{\sqrt{5}}\phi^{n-1}$, but 
$b^{\scalebox{0.7}{$\twocatstuff{F}$},\obstuff{X}}(n)$ is equal to the $(n+1)$th 
Fibonacci number and hence the off-by-one-error in the exponent.
\end{Example}

\begin{Example}[Verlinde category]\label{E:ExamplesVerlinde}
We now consider the \emph{Verlinde category} $\verlinde$ for $k\in\Z_{\geq 2}$, see for example \cite[Section 8.18.2]{EtGeNiOs-tensor-categories} (denoted differently therein). This fusion category has $k$ simple objects, and we 
take the generating object $\obstuff{X}$ of categorical dimension 
$2\cos\big(\pi/(k+1)\big)$. The case $k=2$ compares to super vector spaces.

The action matrix for $\obstuff{X}$ has the type A Dynkin diagram as its associated graph, and the eigenvalues and eigenvectors of this graph are well-known, see 
for example \cite{Sm-ADE}. In particular, $\pfdim\obstuff{X}=2\cos\big(\pi/(k+1)\big)$.
Let $q=\exp\big(\pi i/(k+1)\big)$.
Then {\mt} gives us
\begin{gather*}
\mainsymbol(n)=
\begin{cases}
\frac{[1]_{q}+\dots+[k]_{q}}{[1]_{q}^{2}+\dots+[k]_{q}^{2}}\cdot\big(2\cos(\pi/(k+1))\big)^{n} & \text{if $k$ is even},
\\
\Big(\frac{[1]_{q}+\dots+[k]_{q}}{[1]_{q}^{2}+\dots+[k]_{q}^{2}}\cdot 1 + \frac{[1]_{q}-[2]_{q}+\dots-[k-1]_{q}+[k]_{q}}{[1]_{q}^{2}+\dots+[k]_{q}^{2}}\cdot (-1)^n\Big)\cdot\big(2\cos(\pi/(k+1))\big)^{n} & \text{if $k$ is odd}.
\end{cases}
\end{gather*}
Here $[a]_{q}$ denotes the $a$th quantum number evaluated at $q$. We get, for example:
\begin{gather*}
k=4\colon
\begin{tikzpicture}[anchorbase]
\node at (0,0) {\includegraphics[height=2.7cm]{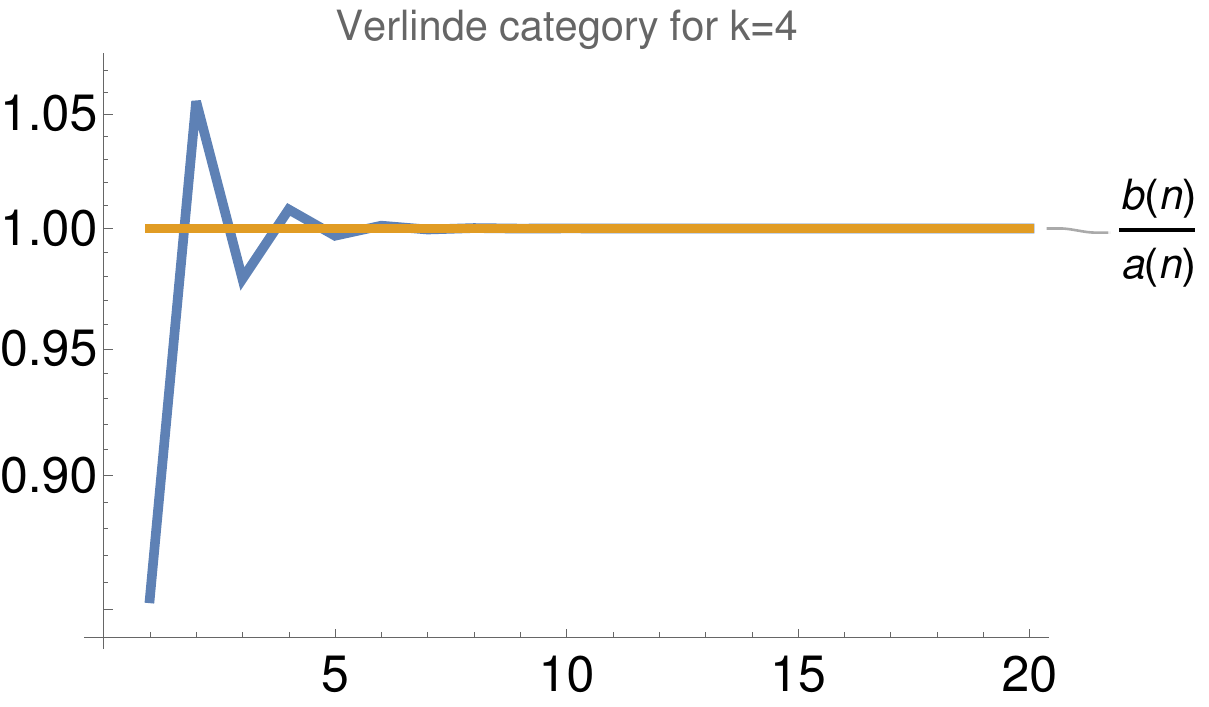}};
\end{tikzpicture}
,\quad
k=6\colon
\begin{tikzpicture}[anchorbase]
\node at (0,0) {\includegraphics[height=2.7cm]{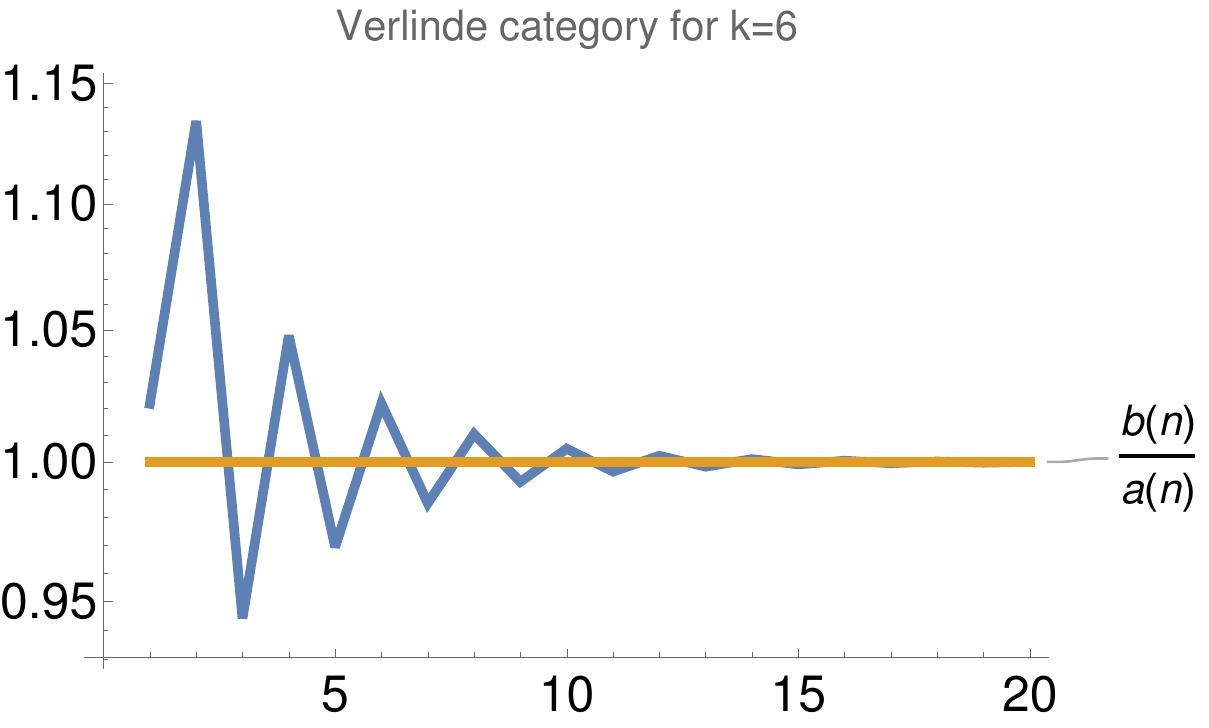}};
\end{tikzpicture}
,\\
k=7\colon
\begin{tikzpicture}[anchorbase]
\node at (0,0) {\includegraphics[height=2.7cm]{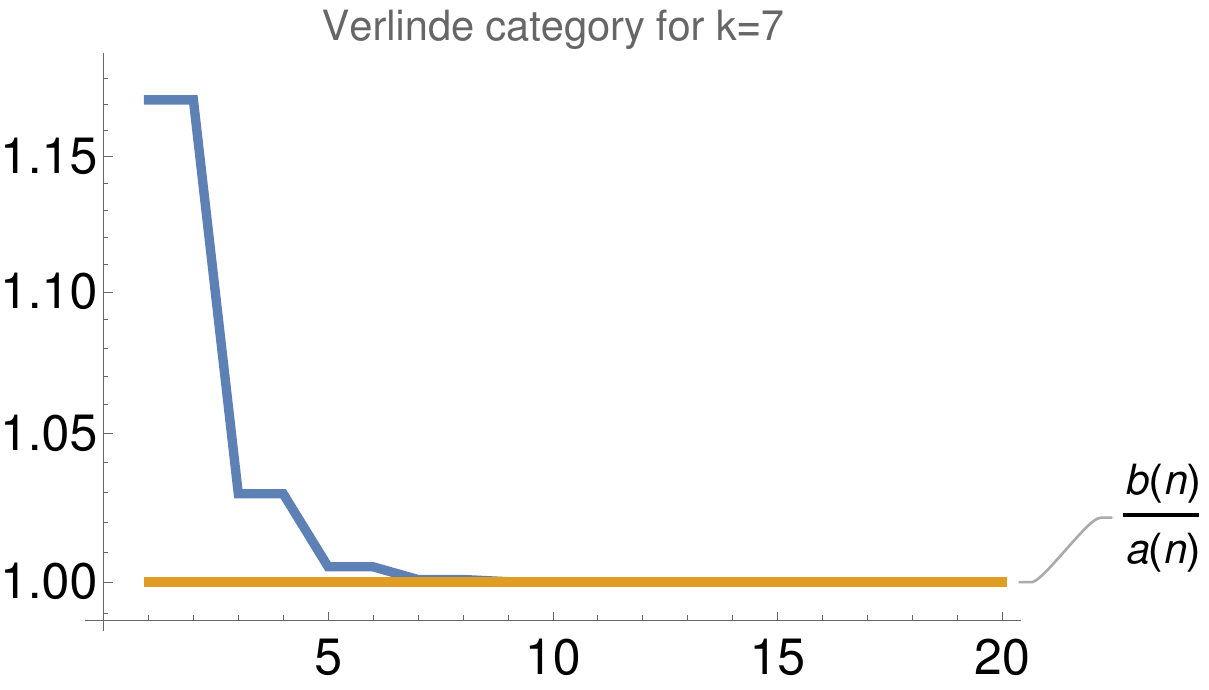}};
\end{tikzpicture}
,\quad
k=9\colon
\begin{tikzpicture}[anchorbase]
\node at (0,0) {\includegraphics[height=2.7cm]{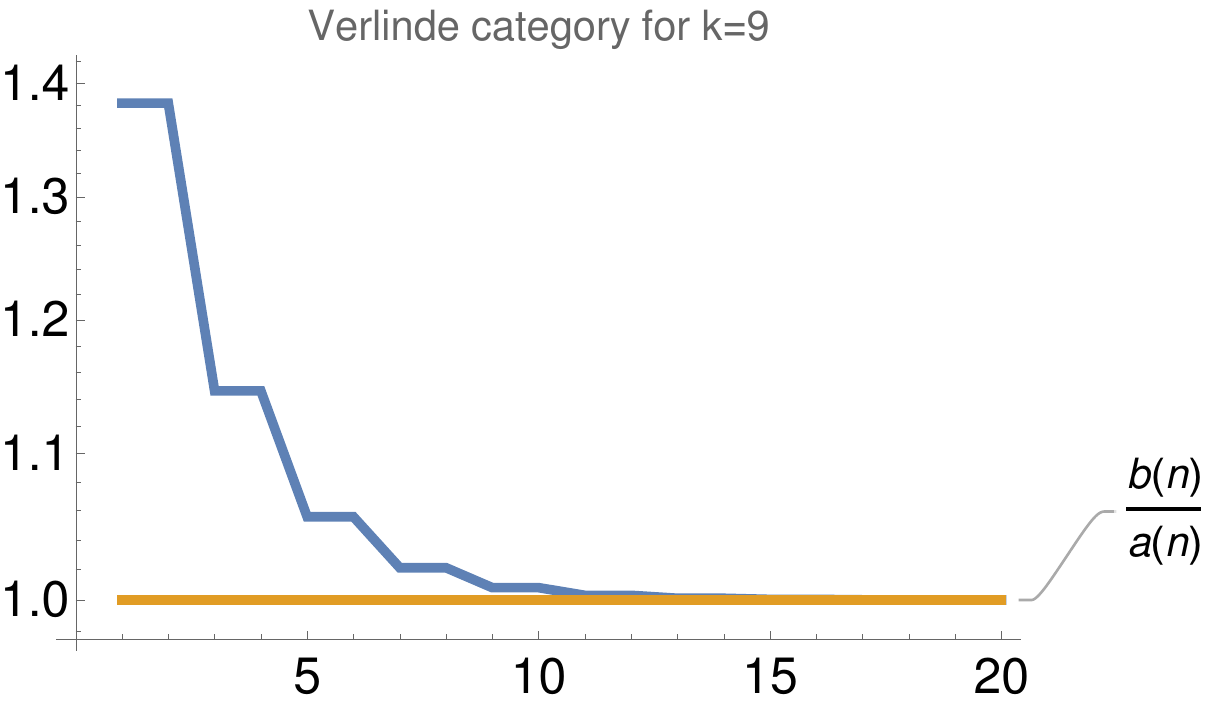}};
\end{tikzpicture}
.
\end{gather*}
Moreover, for $k\in\{3,5\}$ the formula $\mainsymbol(n)$ is spot on.
\end{Example}

\begin{Example}[Higher rank Verlinde categories]\label{E:ExamplesVerlindeTwo}
Verlinde categories can be defined for all simple Lie algebras as 
quotients of representations of quantum groups at a root of unity 
as explained in \cite{AnPa-fusion-lie-algebras}.
Let us focus in this example on $\verlindeg$, the one for the special 
linear group of rank three (with $k$ determined as $e$ in \cite[Section 2]{MaMaMiTu-trihedral}).

For $\verlindeg$ we take the generating object $\obstuff{X}$ corresponding to
the vector representation of $\mathrm{SL}_{3}(\C)$. Its action matrix 
is the oriented version of the graph displayed in \cite[Fig. A1]{MaMaMiTu-trihedral} with the orientation as in \cite[(3-1)]{MaMaMiTu-trihedral}. Using this, and omitting $k=1$ since this is trivial, {\mt} gives
\begin{gather*}
\scalebox{0.97}{$k=2\colon
\mainsymbol(n)=\frac{1}{10}(\sqrt{5}+5)\cdot\phi^{k}
,\quad
k=3\colon
\mainsymbol(n)=
1/2\cdot 2^{n}
,\quad
k=4\colon
\mainsymbol(n)=
\frac{1}{7}\Big(2+2\cos\big(\frac{3\pi}{7}\big)\Big)\cdot\Big(1+2\cos\big(\frac{2\pi}{7}\big)\Big)^{n}$}
,
\\	
k=2\colon
\begin{tikzpicture}[anchorbase]
\node at (0,0) {\includegraphics[height=2.7cm]{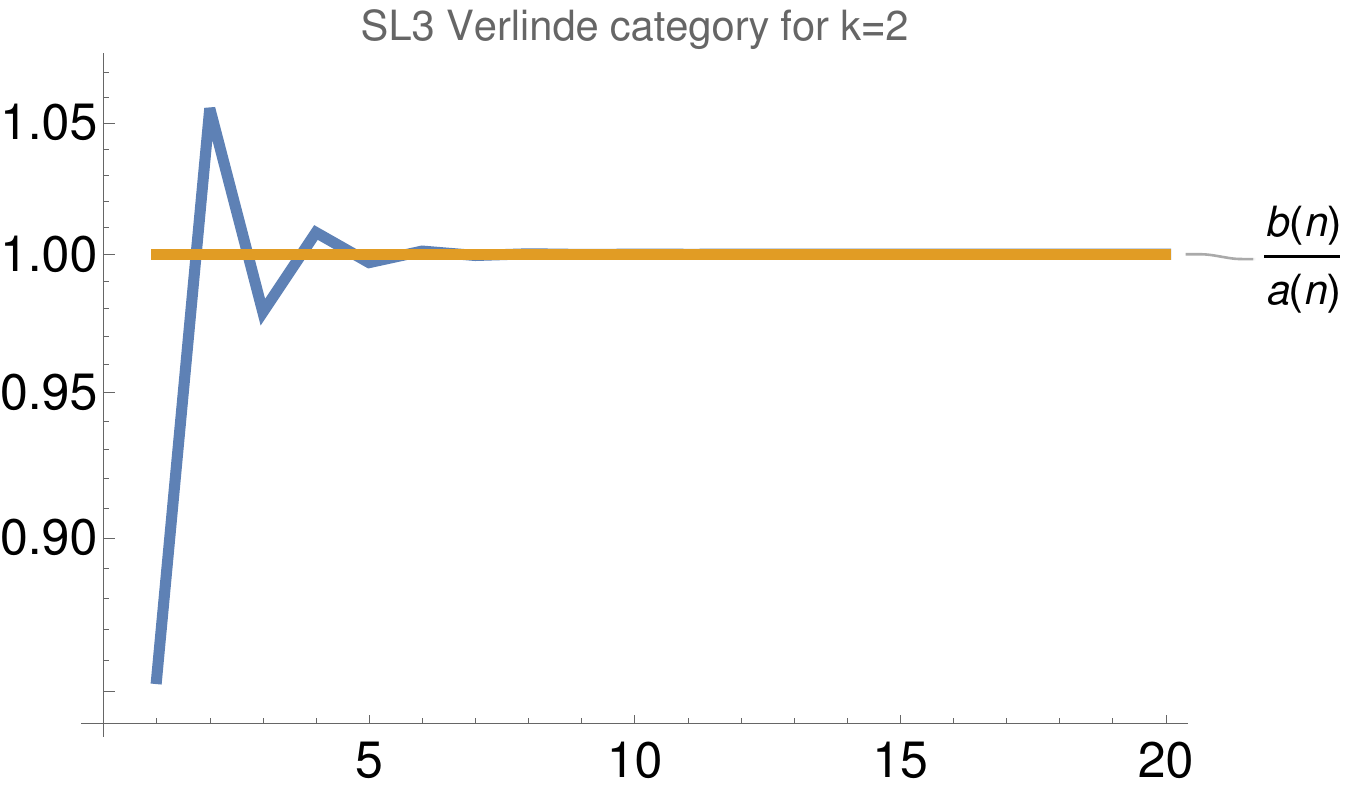}};
\end{tikzpicture}
,\quad
k=4\colon
\begin{tikzpicture}[anchorbase]
\node at (0,0) {\includegraphics[height=2.7cm]{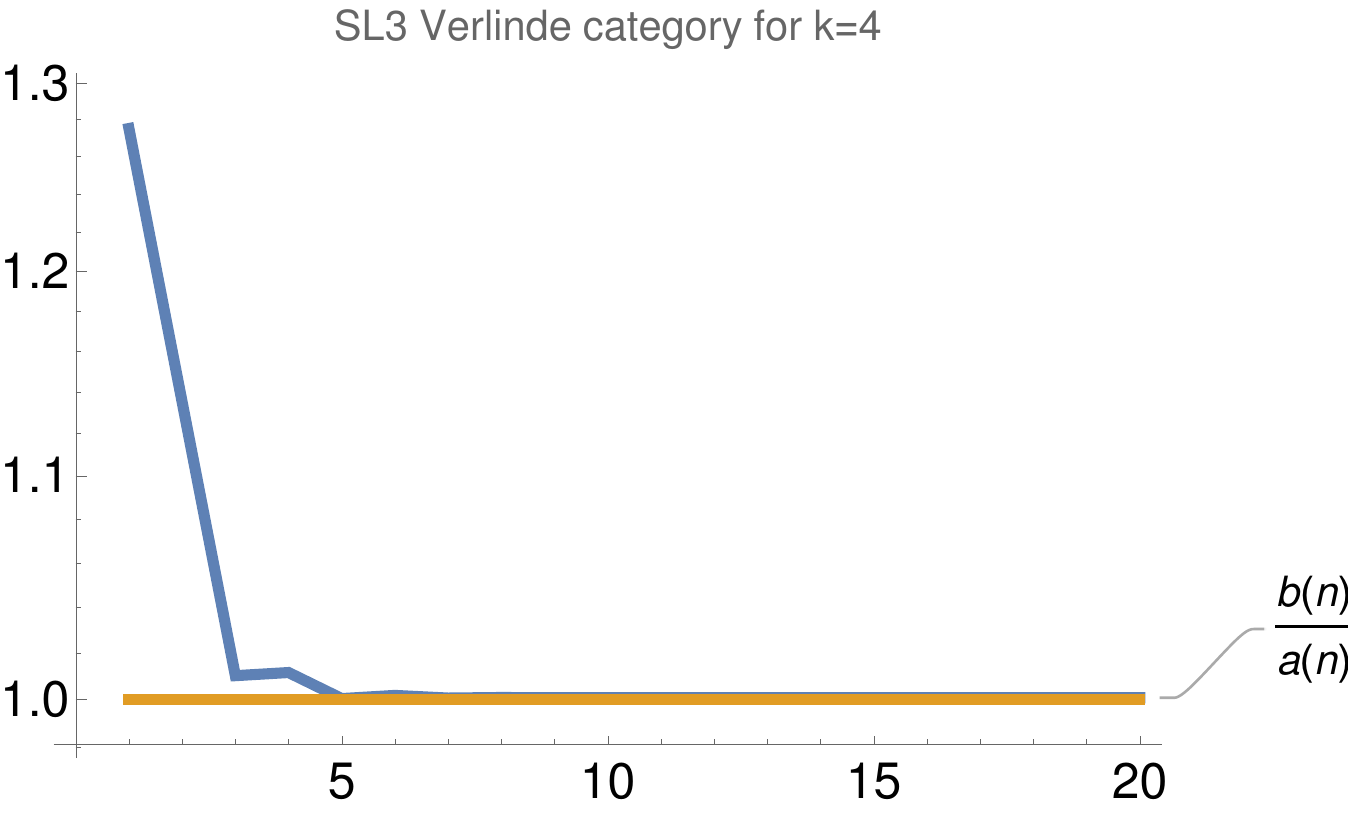}};
\end{tikzpicture}
.
\end{gather*}
For $k=3$ the displayed formulas are exact.
Moreover, one can find $\mainsymbol(n)$ explicitly in general using the formulas 
in \cite{Zu-gen-dynkin-diagrams} or \cite[Section 2]{MaMaMiTu-trihedral}.
\end{Example}

%%%%%%%%%%%%%%%%%%%%%%%%%%%%%%%%%%%%%%%%%

\subsection{Nonsemisimple examples}\label{SS:ExamplesNonsemisimple}

%%%%%%%%%%%%%%%%%%%%%%%%%%%%%%%%%%%%%%%%%

We now discuss two nonsemisimple examples.

\begin{Example}[$\sltwo$ in defining characteristic]\label{E:ExamplesSLTwo}
Let our ground field be $\F$ for some prime $p>2$. We consider 
the finite group $\sltwo$ and its representations over $\F$.
Take $V=\F^{2}$ to be the vector representation of 
$\sltwo$. In this case the action matrices are exemplified by
\begin{gather*}
p=3\colon
\begin{psmallmatrix}
0 & 1 & 0 & 0 & 0 \\
1 & 0 & 0 & 0 & 0 \\
0 & 1 & 0 & 3 & 0 \\
0 & 0 & 1 & 0 & 1 \\
0 & 0 & 0 & 1 & 0 \\
\end{psmallmatrix}
,\quad
p=5\colon
\begin{psmallmatrix}
0 & 1 & 0 & 0 & 0 & 0 & 0 & 0 & 0 \\
1 & 0 & 1 & 0 & 0 & 0 & 0 & 0 & 0 \\
0 & 1 & 0 & 1 & 0 & 0 & 0 & 0 & 0 \\
0 & 0 & 1 & 0 & 0 & 0 & 0 & 0 & 0 \\
0 & 0 & 0 & 1 & 0 & 2 & 0 & 1 & 0 \\
0 & 0 & 0 & 0 & 1 & 0 & 1 & 0 & 0 \\
0 & 0 & 0 & 0 & 0 & 1 & 0 & 1 & 0 \\
0 & 0 & 0 & 0 & 0 & 0 & 1 & 0 & 1 \\
0 & 0 & 0 & 0 & 0 & 0 & 0 & 1 & 0 \\
\end{psmallmatrix}
,\quad
p=7\colon
\begin{psmallmatrix}
0 & 1 & 0 & 0 & 0 & 0 & 0 & 0 & 0 & 0 & 0 & 0 & 0 \\
1 & 0 & 1 & 0 & 0 & 0 & 0 & 0 & 0 & 0 & 0 & 0 & 0 \\
0 & 1 & 0 & 1 & 0 & 0 & 0 & 0 & 0 & 0 & 0 & 0 & 0 \\
0 & 0 & 1 & 0 & 1 & 0 & 0 & 0 & 0 & 0 & 0 & 0 & 0 \\
0 & 0 & 0 & 1 & 0 & 1 & 0 & 0 & 0 & 0 & 0 & 0 & 0 \\
0 & 0 & 0 & 0 & 1 & 0 & 0 & 0 & 0 & 0 & 0 & 0 & 0 \\
0 & 0 & 0 & 0 & 0 & 1 & 0 & 2 & 0 & 0 & 0 & 1 & 0 \\
0 & 0 & 0 & 0 & 0 & 0 & 1 & 0 & 1 & 0 & 0 & 0 & 0 \\
0 & 0 & 0 & 0 & 0 & 0 & 0 & 1 & 0 & 1 & 0 & 0 & 0 \\
0 & 0 & 0 & 0 & 0 & 0 & 0 & 0 & 1 & 0 & 1 & 0 & 0 \\
0 & 0 & 0 & 0 & 0 & 0 & 0 & 0 & 0 & 1 & 0 & 1 & 0 \\
0 & 0 & 0 & 0 & 0 & 0 & 0 & 0 & 0 & 0 & 1 & 0 & 1 \\
0 & 0 & 0 & 0 & 0 & 0 & 0 & 0 & 0 & 0 & 0 & 1 & 0 \\
\end{psmallmatrix}
.
\end{gather*}
These can be described as follows. 
The matrix is the one obtained as a $(2p-1)$-by-$(2p-1)$ cut-off of the matrix 
for the infinite group over $\overline{\F}$ that can be obtained from \cite[Proposition 4.4]{SuTuWeZh-mixed-tilting}, together with an extra entry $1$ in position $(p,2p-2)$.

Then {\mt} gives
\begin{gather*}
\mainsymbol(n)=\left(\frac{1}{2p-2}\cdot 1+\frac{1}{2p^{2}-2p}\cdot(-1)^{n}\right)\cdot 2^{n}.
\end{gather*}
Explicitly, for $p\in\{3,5\}$ we get
\begin{gather*}
p=3\colon
\begin{tikzpicture}[anchorbase]
\node at (0,0) {\includegraphics[height=2.7cm]{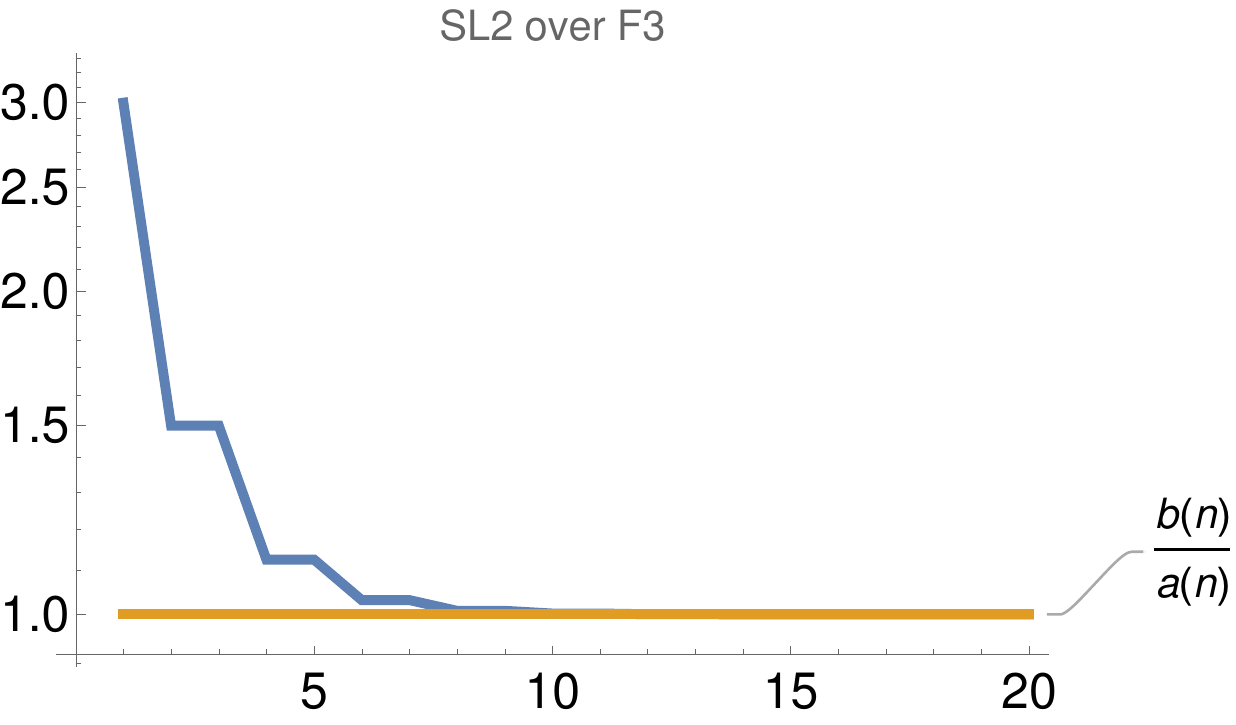}};
\end{tikzpicture}
,\quad
p=5\colon
\begin{tikzpicture}[anchorbase]
\node at (0,0) {\includegraphics[height=2.7cm]{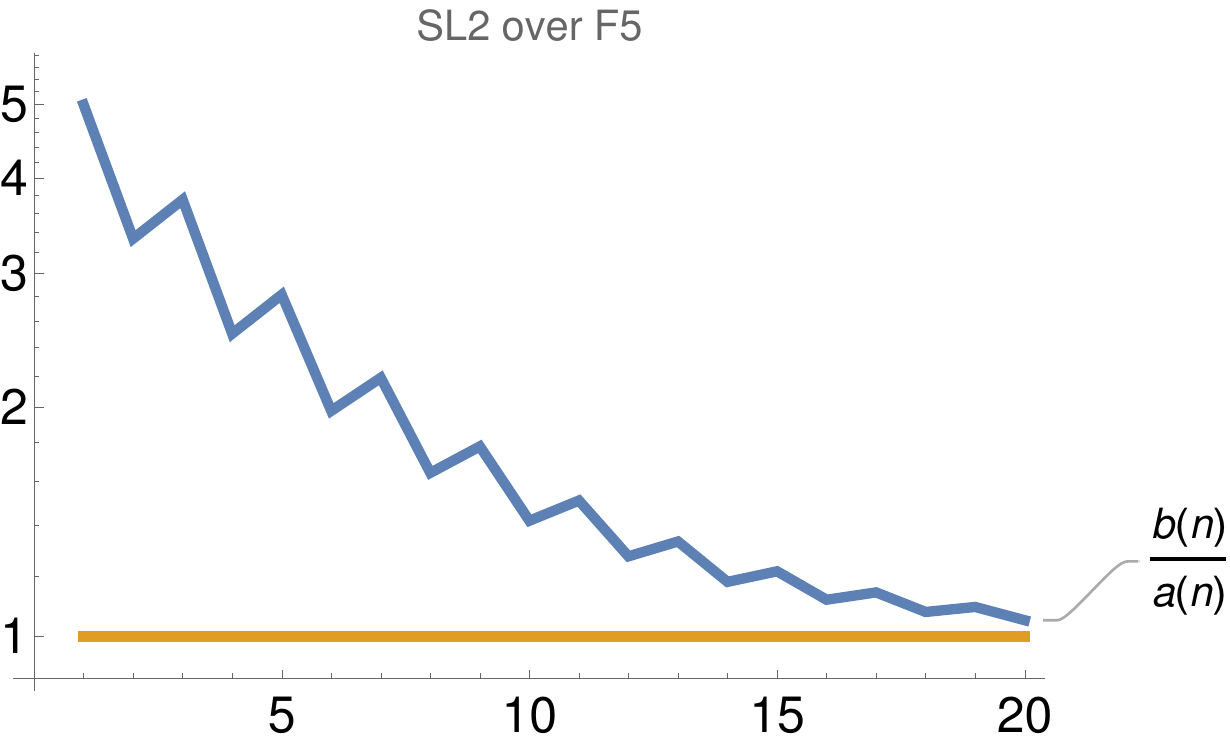}};
\end{tikzpicture}
.
\end{gather*}
The convergence is rather slow (but still geometric).
\end{Example}

\begin{Example}[Dihedral Soergel bimodules]\label{E:ExamplesSBim}
We now look at the 
category of dihedral Soergel bimodules 
as studied in details in, for example, \cite{El-two-color-soergel}, 
\cite{MaTu-soergel} or \cite{Tu-sandwich-cellular}.
In particular, \cite[Section 3C]{Tu-sandwich-cellular} lists all the formulas 
relevant for {\mt}. To get a finite based $\Rplus$-algebra we collapse the grading, meaning 
we specialize \cite[Section 3C]{Tu-sandwich-cellular} at $q=1$.

Fix $\langle s,t|s^{2}=t^{2}=(st)^{m}\rangle$ as the presentation 
for the dihedral group of order $2m$ where $m\in\Z_{\geq 3}$.
Let us take $\obstuff{X}$ to be the Bott--Samelson generator for $st$. 
By the explicit 
formulas in \cite[Section 3C]{Tu-sandwich-cellular},
the action graph of $\obstuff{X}$ is almost the same as the action graph 
of tensoring with $\C^{3}$ as a $\mathrm{SO}_{3}(\C)$-representation. 
The first ones are (read from left to right):
\begin{center}
\begin{tabular}{c||c}
$m\in\{3,5,7\}$	& $m\in\{4,6,8\}$ \\
	\hline
\begin{tikzpicture}[anchorbase]
	\node at (0,0) {\includegraphics[height=1.3cm]{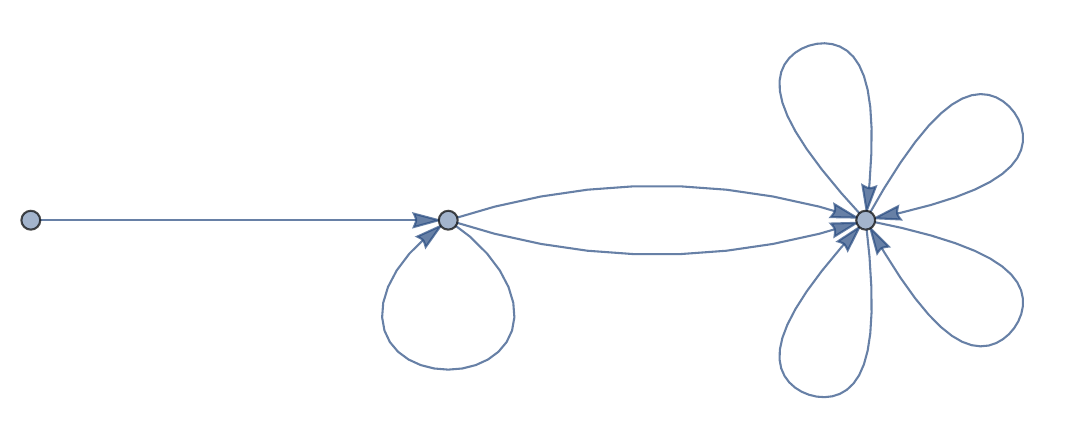}};
\end{tikzpicture}	&  \begin{tikzpicture}[anchorbase]
\node at (0,0) {\includegraphics[height=1.3cm]{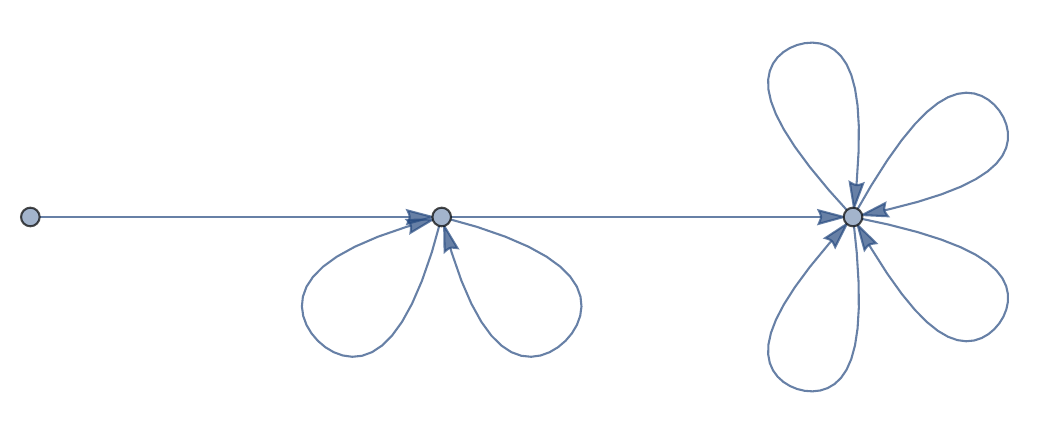}};
\end{tikzpicture} \\
	\hline
\begin{tikzpicture}[anchorbase]
	\node at (0,0) {\includegraphics[height=1.3cm]{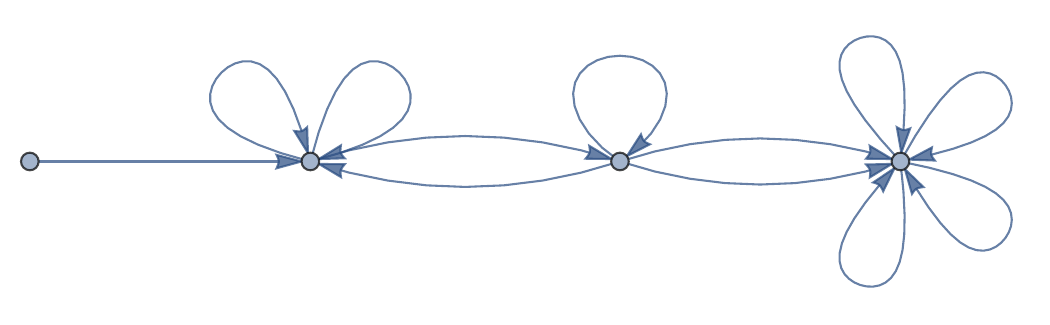}};
\end{tikzpicture}	&  \begin{tikzpicture}[anchorbase]
\node at (0,0) {\includegraphics[height=1.3cm]{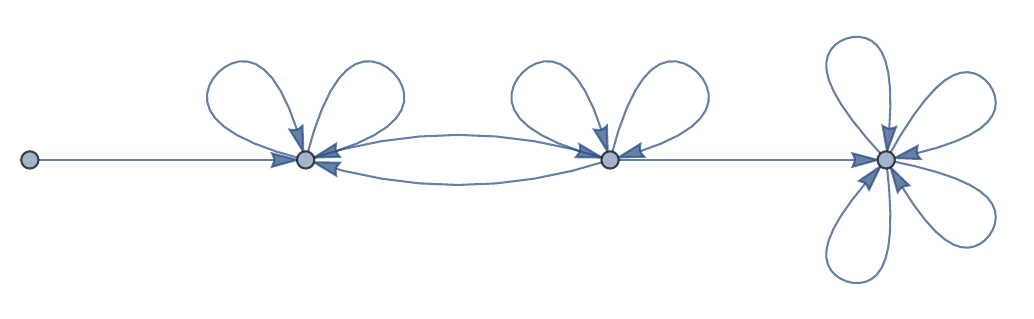}};
\end{tikzpicture} \\
	\hline
\begin{tikzpicture}[anchorbase]
	\node at (0,0) {\includegraphics[height=1.3cm]{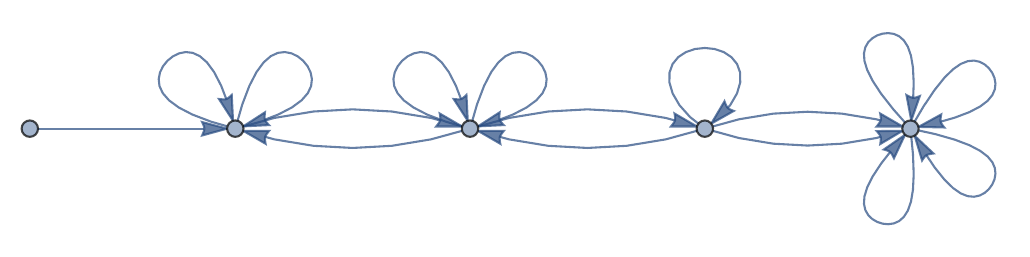}};
\end{tikzpicture}	&  \begin{tikzpicture}[anchorbase]
\node at (0,0) {\includegraphics[height=1.3cm]{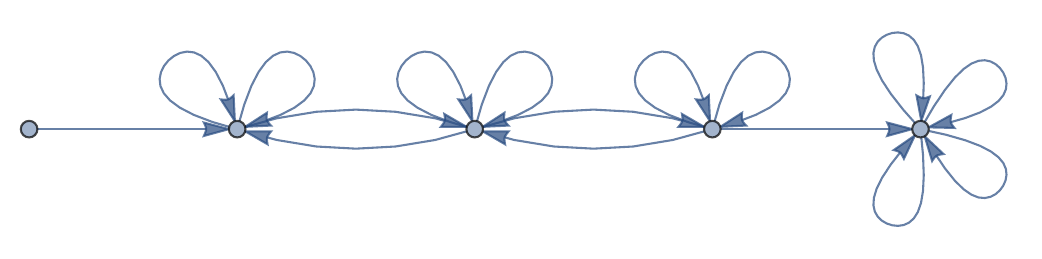}};
\end{tikzpicture} \\
\end{tabular}
\end{center}
The pattern generalizes. It is then easy to show that 
the leading eigenvalues is always $4$ and 
the absolute values of all other eigenvalues 
are strictly smaller. Moreover, {\mt} gives: 
\begin{gather*}
\mainsymbol(n)=\frac{1}{2m}\cdot 4^{n},
\\
m=3\colon
\begin{tikzpicture}[anchorbase]
\node at (0,0) {\includegraphics[height=2.7cm]{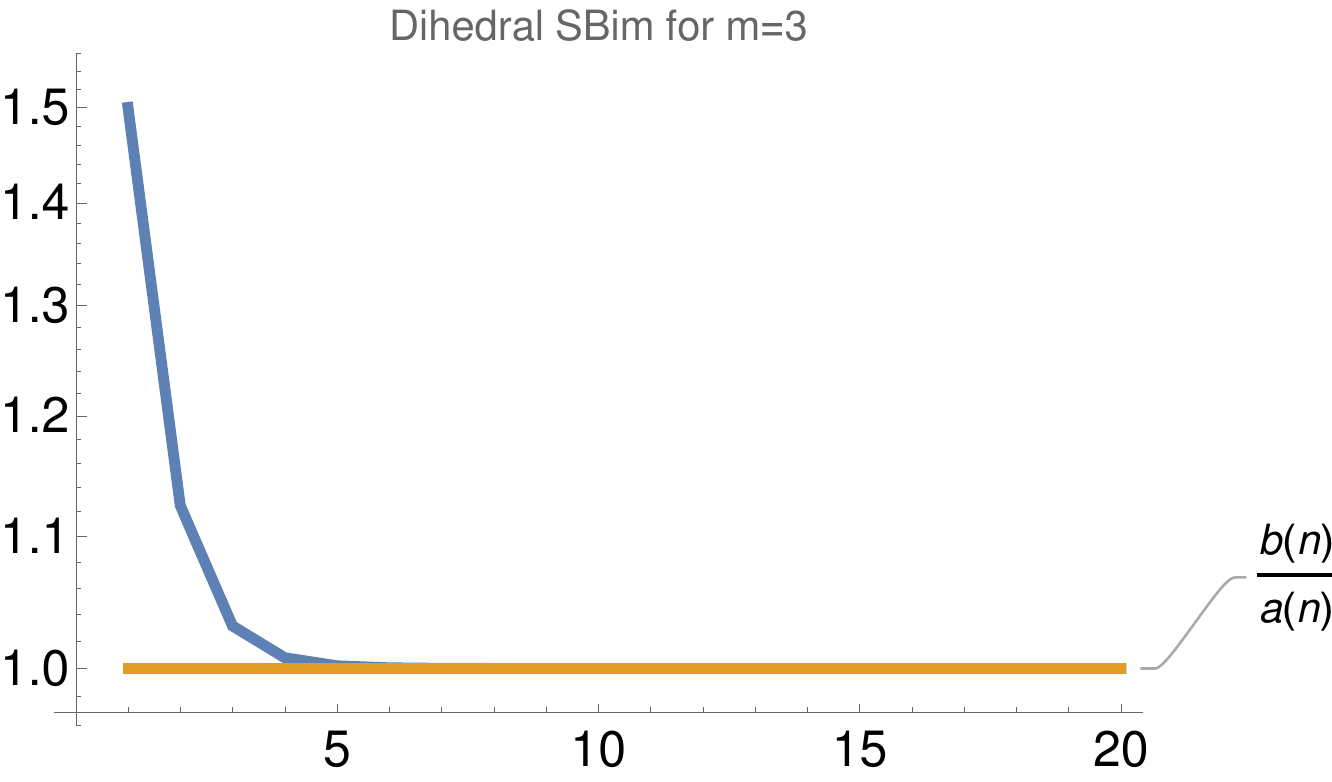}};
\end{tikzpicture}
,\quad
m=7\colon
\begin{tikzpicture}[anchorbase]
\node at (0,0) {\includegraphics[height=2.7cm]{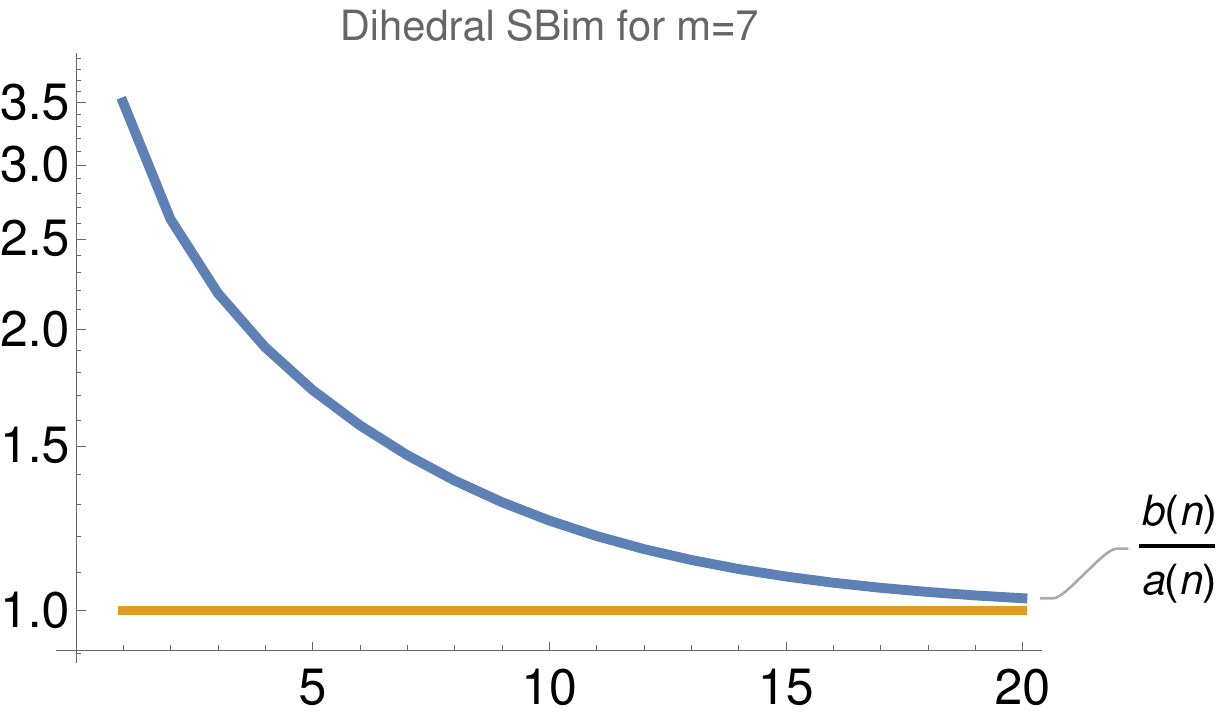}};
\end{tikzpicture}
.
\end{gather*}
The rate of convergence is rather slow for $m\gg 0$.
\end{Example}

%%%%%%%%%%%%%%%%%%%%%%%%%%%%%%%%%%%%%%%%%

\section{Generalizations and proofs}\label{S:Proof}

%%%%%%%%%%%%%%%%%%%%%%%%%%%%%%%%%%%%%%%%%

We will prove several versions of \autoref{T:IntroMain}.

%%%%%%%%%%%%%%%%%%%%%%%%%%%%%%%%%%%%%%%%%

\subsection{Perron--Frobenius theory}\label{SS:ProofPF}

%%%%%%%%%%%%%%%%%%%%%%%%%%%%%%%%%%%%%%%%%

We start with the main player, the \emph{Perron--Frobenius theorem}. 
To this end, recall that one can associate 
an oriented and weighted graph, its \emph{adjacency graph}, to an $m$-by-$m$ matrix 
$M=(m_{ij})_{1\leq i,j\leq m}\in\rmat$ as follows:
\begin{enumerate}[label=(\roman*)]

\item The vertices are $\{1,\dots,m\}$.

\item There is an edge with weight $m_{ij}$ from $i$ to $j$.

\end{enumerate}
We call a nonzero matrix 
$M\in\rmat$ \emph{irreducible} if its associated graph
is connected in the oriented sense (this is called \emph{strongly connected}).
Recall that in this note a right eigenvector satisfies $Mv=\lambda\cdot v$, 
and a left eigenvector satisfies $w^{T}M=\lambda\cdot w^{T}$.

\begin{Theorem}[Perron--Frobenius theorem part I]\label{T:ProofPFTheorem}
Let $M\in\rmat$ be irreducible.
\begin{enumerate}

\item $M$ has a \emph{Perron--Frobenius eigenvalue}, that is, 
$\lambda\in\R_{>0}$ such that $\lambda\geq|\mu|$ for all other eigenvalues $\mu$.
This eigenvalue appears with multiplicity one, and all other eigenvalues 
with $\lambda=|\mu|$ also appear with multiplicity one.

\item There exists $h\in\Z_{\geq 1}$, the \emph{period}, such that all eigenvalues 
$\mu$ with $\lambda=|\mu|$ are $\exp(k2\pi i/h)\lambda$ for $k\in\{0,\dots,h-1\}$.
We call these \emph{pseudo-dominant eigenvalues}.

\item The eigenvectors, left and right, for the Perron--Frobenius eigenvalue can be normalized to have entries in $\Rplus$.

\end{enumerate}
\end{Theorem}

\begin{proof}
Well-known. See for example, Frobenius' paper 92 in Band 3 of
\cite{Fr-werke}. (This is the paper ``{\"U}ber Matrizen aus nicht negativen Elementen''.)
\end{proof}

Fix a function $f\colon\N\to\N$. We say that $f(n)$ converges \emph{geometrically 
to $a\in\R$ with ratio $\beta\in[0,1)$} if for all $\gamma\in(\beta,1)$ we have that $\{(f(n)-a)/\gamma^{n}\}_{n\in\N}$ is bounded. We, abusing language, will call the infimum of all ratios \emph{the} ratio of convergence.

For two matrices of the same size let $\sim$ mean that they are asymptotically equal entrywise 
(note that below $v_{i}w_{i}^{T}$ are matrices). For such matrices we apply 
the definition of geometric convergence entrywise with the ratio being the maximum of the entrywise ratios.
The following accompanies \autoref{T:ProofPFTheorem}:

\begin{Theorem}[Perron--Frobenius theorem part II]\label{T:ProofPFTheoremTwo}
Let $M\in\rmat$ be irreducible, $\lambda$ be its Perron--Frobenius eigenvalue and $h$ be its period. Let $\zeta=\exp(2\pi i/h)$.
For each $k\in\{0,\dots,h-1\}$, choose a left eigenvector $v_{k}$ and a right eigenvector $w_{k}$ with eigenvalue $\zeta^{k}\lambda$, normalized such that $w_{k}^{T}v_{k}=1$.

Then we have:
\begin{gather*}
M^{n}\sim v_{0}w_{0}^{T}\cdot\lambda^{n}
+v_{1}w_{1}^{T}\cdot(\zeta\lambda)^{n}
+v_{2}w_{2}^{T}\cdot(\zeta^2\lambda)^{n}+\dots
+v_{h-1}w_{h-1}^{T}\cdot(\zeta^{h-1}\lambda)^{n}.
\end{gather*}
Moreover, the convergence is geometric with ratio $|\lambda^{sec}/\lambda|$, where 
$\lambda^{sec}$ is any second largest (in the sense of absolute value) eigenvalue.
\end{Theorem}

\begin{proof}
This is known, but proofs are a bit tricky to find in the literature, so we give one. 
The proof also shows where the vectors $v_{i}$ and $w_{i}$ come from.

For any $\mu\in\C$, let $V_{\mu}$ be the generalized eigenspace of $V=\C^{m}$ associated to the eigenvalue $\mu$. Then we have
\begin{gather*}
V=
\bigoplus_{k=0}^{h}V_{\zeta^{k}\lambda}\oplus
\bigoplus_{\mu,|\mu|<\lambda}V_{\mu}.
\end{gather*}
By \autoref{T:ProofPFTheorem}, the space $V_{\zeta^{k}\lambda}$ is the eigenspace associated to the eigenvalue $\zeta^{k}\lambda$ and $v_{k}w_{k}^{T}$ is the projection onto that subspace.

This implies that we have
\begin{gather*}
M^{n}=v_{0}w_{0}^{T}\cdot\lambda^{n}
+v_{1}w_{1}^{T}\cdot(\zeta\lambda)^{n}
+v_{2}w_{2}^{T}\cdot(\zeta^2\lambda)^{n}+\dots
+v_{h-1}w_{h-1}^{T}\cdot(\zeta^{h-1}\lambda)^{n}
+R(n),
\end{gather*}
where $R(n)$ is the multiplication action of $M^{n}$ onto the rest. Since the eigenvalues $\mu$ of $M$ on the rest satisfies $|\mu|<\lambda$, we have $R(n)/\lambda^{n}\to_{n\to\infty}0$ geometrically with ratio $|\lambda^{sec}/\lambda|$.
\end{proof}

For a general matrix $M\in\rmat$ things change, but not too much:

\begin{Theorem}[Perron--Frobenius theorem part III]\label{T:ProofPFTheoremThree}
Let $M\in\rmat$.
\begin{enumerate}

\item $M$ has a \emph{Perron--Frobenius eigenvalue}, that is, 
$\lambda\in\Rplus$ such that $\lambda\geq|\mu|$ for all other eigenvalues $\mu$.

\item Let $s$ be the multiplicity of the Perron--Frobenius eigenvalue.
There exists $(h_{1},\dots,h_{s})\in\Z_{\geq 1}^{s}$, the \emph{periods}, such that all eigenvalues $\mu$ with $\lambda=|\mu|$ are $\exp(k2\pi i/h_{\ell})\lambda$ for $k\in\{0,\dots,h_{\ell}-1\}$, for some period. We call these \emph{pseudo-dominant eigenvalues}.

\item The eigenvectors, left and right, for the Perron--Frobenius eigenvalues can be normalized to have entries in $\Rplus$.

\item Let $h=\mathrm{lcm}(h_{1},\dots,h_{s})$, and 
let $\nu$ the maximal dimension of the Jordan blocks of $M$ containing $\lambda$.
There exist matrices 
$S^{i}(n)$ with polynomial entries of degree $\leq(\nu-1)$ for $i\in\{0,\dots,h-1\}$ such that
\begin{gather*}
\lim_{n\to\infty}|\left(M/\lambda\right)^{hn+i}-S^{i}(n)|\to 0
\quad\forall i\in\{0,\dots,h-1\},
\end{gather*}
and the convergence is geometric with ratio $|\lambda^{sec}/\lambda|^{h}$.
There are also explicit formulas for the matrices $S^{i}(n)$, see \cite[Section 5]{Ro-expansion-sums-matrix-powers}.
\end{enumerate}

\end{Theorem}

\begin{proof}
This can be found in \cite{Ro-expansion-sums-matrix-powers}. See also 
\cite[Section I.10]{Ho-handbook-linear-algebra} (in the second version) for a useful list of properties of nonnegative matrices.
\end{proof}

%%%%%%%%%%%%%%%%%%%%%%%%%%%%%%%%%%%%%%%%%

\subsection{Three versions of the {\mt}}\label{SS:ProofVersions}

%%%%%%%%%%%%%%%%%%%%%%%%%%%%%%%%%%%%%%%%%

We recall based algebras. These algebras originate in work of Lusztig on 
so-called \emph{special} representations of Weyl groups \cite{Lu-irreps-weyl-I}.
We follow \cite[Section 2]{KiMa-based-algebras} with our definition.

Let $\K\subset\C$ be a unital subring. A $\K$-algebra $R$ with a finite 
$\K$-basis $C=\{1=c_{0},\dots,c_{r-1}\}$ is called a \emph{finite 
based $\Rplus$-algebra} if all structure constants are in $\Rplus$ 
with respect to the basis $C$. That is, \eqref{Eq:IntroBased} holds.

The underlying ring $\K$ is allowed to be different from $\R$ or $\C$, but it needs to contain the structure constants of course. When the structure constants are in $\N\subset\Rplus$ a popular choice for the ground ring is $\K=\Z$.

\begin{Example}\label{E:ProofBased}
Examples include:
\begin{enumerate}

\item The Grothendieck rings of all the examples in \autoref{E:IntroFiniteCats}. In these examples one often takes $\K=\Z$, but other rings are allowed as well.

\item Group or more general semigroup algebras for finite groups or semigroups.

\item There are many interesting infinite examples coming from 
skein theory, see {\eg} \cite{Th-positive-skein}.	

\end{enumerate}
Decategorifications are our main examples where $\Rplus$ 
can be replaced by $\N$.
\end{Example}

A finite based $\Rplus$-algebra is actually a pair $(R,C)$, but we will 
write $R$ for short. Next, fix such an $R$ and $c\in\Rplus C$.
In this setting we can define the \emph{(pre) action matrix} $M^{\prime}(c)_{k,j}=\sum_{i}a_{i}m_{i,j}^{k}\in\Rplus$. 
The \emph{action matrix} $M(c)$ is then the adjacency matrix for the 
connected component, in the nonoriented sense, of the identity $1\in C$ in the adjacency graph of $M^{\prime}(c)$.
Note that $M(c)\in\rmat$ is a submatrix of $M^{\prime}(c)\in\rmat[r]$ 
for some $1\leq m\leq r$.

We give three versions of {\mt}, stated in terms of finite based $\Rplus$-algebras. 
The categorical version then follows immediately from \autoref{L:IntroGrothendieck}.

\begin{Theorem}[Version 1]\label{T:ProofMainOne}
Fix a finite based $\Rplus$-algebra $R$, and $c$ a $\Rplus$-linear combination of elements 
from $C$. Assume that the action matrix $M(c)$ is \textbf{irreducible}.
Then \autoref{T:IntroMain} holds with $\mainsymbol(n)$ as in \eqref{Eq:IntroMainSymbol}.
\end{Theorem}

\begin{proof}
Consider the following matrix equation:
\begin{gather*}
M(c)c(n-1)=c(n),
\end{gather*}
where $c(k)=\big(c_{0}(k),\dots,c_{r-1}(k)\big)\in\Rplus^{r}$ are 
vectors such that their $i$th entry is the 
multiplicity of $c_{i}$ in $c^{k}$, and $c(0)=(1,0,\dots,0)^{T}$ with the 
one is in the slot of $c_{0}=1$. This equation holds by 
the definition of the action matrix.
Iterating this process, we get
\begin{gather*}
M(c)^{n}c(0)=c(n).
\end{gather*}
Note that $M(c)^{n}c(0)$ is the same as taking the first column of $M(c)^{n}$. 
Hence,
\begin{gather*}
b^{R,c}(n)=M(c)^{n}[1]
\end{gather*}
in the notation of the introduction. Thus, \autoref{T:ProofPFTheoremTwo} implies the result.
\end{proof}

\begin{Remark}\label{R:ProofMainOne}
\autoref{T:ProofMainOne} is sufficient for many example. Explicitly, 
\autoref{T:ProofMainOne} works for all \emph{transitive} finite based $\Rplus$-algebras.
Examples include all finite monoidal categories that are rigid by 
\cite[Proposition 4.5.4]{EtGeNiOs-tensor-categories}.
\end{Remark}

We say that $M\in\rmat$ has the \emph{Perron--Frobenius property} 
if its Perron--Frobenius eigenvalue has multiplicity one.

\begin{Theorem}[Version 2]\label{T:ProofMainTwo}
Fix a finite based $\Rplus$-algebra $R$, and $c$ an $\Rplus$-linear combination of elements 
from $C$. Assume that the action matrix $M(c)$ has the \textbf{Perron--Frobenius property}.
Then \autoref{T:IntroMain} holds with $\mainsymbol(n)$ as in \eqref{Eq:IntroMainSymbol}.
\end{Theorem}

\begin{proof}
The iteration works as in the proof of \autoref{T:ProofMainOne}, so let us focus on the growth rate. We will use \autoref{T:ProofPFTheoremThree} for $s=1$. This implies 
that $\nu=1$, by its definition. In particular, we only have
$S^{i}(n)$ with entries of degree zero, so these are matrices that do not depend on $n$, so we can simply write $S^{i}$. We will argue that they 
are essentially the matrices $v_{i}w_{i}^{T}$.

Precisely, as follows from \cite[Section 5]{Ro-expansion-sums-matrix-powers}, we have
\begin{align*}
S^{i}
&=
v_{0}w_{0}^{T}
+v_{1}w_{1}^{T}\cdot\zeta^{i}
+v_{2}w_{2}^{T}\cdot\zeta^{2i}+\dots
+v_{h-1}w_{h-1}^{T}\cdot\zeta^{(h-1)i}
\\
&=
v_{0}w_{0}^{T}
+v_{1}w_{1}^{T}\cdot\zeta^{nh+i}
+v_{2}w_{2}^{T}\cdot\zeta^{2(nh+i)}+\dots
+v_{h-1}w_{h-1}^{T}\cdot\zeta^{(h-1)(nh+i)i}
.
\end{align*}
Now we apply \autoref{T:ProofPFTheoremThree}.(c).
\end{proof}

\begin{Remark}\label{R:ProofMainTwo}
\autoref{T:ProofMainTwo} is the version we used in \autoref{S:Examples}.
\end{Remark}

Recall that the polynomials $S^{i}(n)$ are explicitly given in 
\cite[Section 5]{Ro-expansion-sums-matrix-powers} and define:
\begin{gather}\label{Eq:ProofMainSymbol}
\mainsymbol(n)=
\frac{1}{h}\sum_{i=0}^{h-1}\sum_{j=0}^{h-1}S^{j}\big(\lfloor n/h\rfloor\big)\cdot\zeta^{i(n-j)}
.
\end{gather}

\begin{Theorem}[Version 3]\label{T:ProofMainThree}
Fix a finite based $\Rplus$-algebra $R$, and $c$ an $\Rplus$-linear combination of elements 
from $C$. Then \autoref{T:IntroMain} holds with $\mainsymbol(n)$ as in \eqref{Eq:ProofMainSymbol}.
\end{Theorem}

\begin{proof}
Observing that $1+\zeta^{i}+\dots+\zeta^{(h-1)i}=0$ if $i\not\equiv 0\bmod h$, this follows as for the previous theorems.
\end{proof}

\end{document}